\documentclass[12pt]{amsart}

\usepackage{amssymb}
\usepackage{amsmath}
\numberwithin{equation}{section}
\numberwithin{figure}{section}
\usepackage{float}

\usepackage{color}
\usepackage{graphicx}
\usepackage{diagbox}
\usepackage{array}

\DeclareMathOperator{\sgn}{sgn}
\DeclareMathOperator{\im}{im}
\usepackage[vcentermath,enableskew]{youngtab}
\usepackage{ytableau}
\usepackage{amssymb}
\usepackage{rotating}
\usepackage{young}
\usepackage{enumerate}
\usepackage{hyperref}

\renewcommand{\L}{\cal L}
\newcommand{\no}{\noindent}

\newcommand{\be}{\begin{equation}} 
\newcommand{\bea}{\begin{eqnarray}}
\newcommand{\ee}{\end{equation}}
\newcommand{\beas}{\begin{eqnarray*}}
\newcommand{\eea}{\end{eqnarray}}
\newcommand{\eeas}{\end{eqnarray*}}
\newcommand{\non}{\nonumber}
\newcommand{\cal}{\mathcal}
\newcommand{\lie}{\mathrm{Lie}}

  \newcommand\subsetsim{\,\,\,\small{\mathrel{\substack{
  \textstyle\subset\\[-0.2ex]\textstyle\sim}}}\,\,\,}

\def\C{{\mathbb C}}

\def\z3{{\mathbb Z_3}}

\def\L{{\cal L}}

\def\sg{\mathfrak S}

\newcommand\mdoubleplus{\mathbin{+\mkern-10mu+}}
\def\conc{\mdoubleplus}

\newtheorem{theorem}{Theorem}[section]
\newtheorem{definition}[theorem]{Definition}

\newtheorem{corollary}[theorem]{Corollary}
\newtheorem{conjecture}[theorem]{Conjecture}
\newtheorem{lemma}[theorem]{Lemma}
\newtheorem{proposition}[theorem]{Proposition}
\newtheorem{remark}[theorem]{Remark}
\newtheorem{ex}[theorem]{Example}

\title[$n$-ary generalization of the Lie representation]{On an $n$-ary generalization of the Lie representation  and tree Specht modules}
\author[Friedmann]{Tamar Friedmann}
\address{Department of Mathematics, Colby College}
\email{tfriedma@colby.edu}

\author[Hanlon]{Phil Hanlon}
\address{Department of Mathematics, Dartmouth College}
\email{philip.j.hanlon@dartmouth.edu}

\author[Wachs]{Michelle L. Wachs$^2$}
\address{Department of Mathematics, University of Miami}
\email{wachs@math.miami.edu}
\thanks{$^{2}$Supported in part by NSF grants DMS 1502606 and DMS 2207337}

\begin{document}

\maketitle

\begin{center}
\dedicatory{\em Dedicated to the memory of Adriano Garsia.}
\end{center}

\begin{abstract} We continue our study, initiated in our prior work with Richard Stanley, of the representation of the symmetric group on the multilinear component of an $n$-ary generalization of the free Lie algebra  known as the free Filippov $n$-algebra with $k$ brackets.   Our ultimate aim is to determine the multiplicities of the irreducible representations in this representation.  This had been done for the ordinary Lie representation ($n=2$ case) by Kraskiewicz and Weyman.  The $k=2$ case was handled in our prior work, where the representation was shown to be isomorphic to $S^{2^{n-1}1}$. In this paper, for general $n$ and $k$, we obtain decomposition results that enable us  to determine the multiplicities in the $k=3$ and $k=4$ cases.  In particular we prove that in the $k=3$ case, the representation is isomorphic to $S^{3^{n-1}1} \oplus  S^{3^{n-2}21^2}$.  Our main result shows that  the multiplicities stabilize in a certain sense when $n$ exceeds $k$. As an important tool in proving this, we present  two types of  generalizations of the notion of Specht module that involve trees.
\end{abstract}

\tableofcontents
\section{Introduction} \label{intro}

There are various generalizations of the notion of Lie algebra that involve replacing the binary Lie bracket with an $n$-ary  bracket. One of these was introduced by Filippov  in \cite{Fi} and another was introduced by Hanlon and Wachs in \cite{HW}.  
Both generalizations are of interest in 
elementary particle physics, see e.g. \cite{CZ, DI, Fr, Ji}.
For both the Filippov generalization and the Hanlon-Wachs generalization, the representation of the symmetric group on the multilinear component of the free $n$-ary Lie algebra generalizes an extensively studied representation known as the Lie representation; see  \cite{Ba, BS, Ga, HS, Kl, KW, Re, Wa} for some earlier work on the Lie representation and \cite{Su1, Su2} for  some more recent work.   The representation for the  Hanlon-Wachs generalization  was first studied in \cite{HW}. More recently, a study of  the representation for the  Filippov generalization was initiated   by Friedmann, Hanlon, Stanley, and Wachs in \cite{FHSW1}. 
In this paper, we continue this study, and in particular, answer a question raised in \cite[Question 4.1]{FHSW1} in the affirmative.

Throughout this paper,  all vector spaces are taken over  a   field  of characteristic $0$, say $\C$.  Fix $n \ge 2$.
A {\em Filippov algebra} is a vector space $\L$ equipped with an $n$-linear  map (called a bracket)
\[ [\cdot, \cdot , \hskip .7cm , \cdot] : \times ^n \L \rightarrow \L ~\]
that satisfies the following antisymmetry relation for all $\sigma$ in the symmetric group $\sg_n$:
\begin{equation} \label{antsym} [x_1,\dots,x_n] = \sgn(\sigma) [x_{\sigma(1)},\dots, x_{\sigma(n)}] \end{equation}
and the following generalization of the Jacobi identity:
\begin{align}
\label{type1} &[[x_1, x_2, \ldots,  x_n], x_{n+1},\ldots , x_{2n-1}]
\\ \non&=\sum_{i=1}^{n} [x_1, x_{2}, \ldots , x_{i-1},[ x_{i}, x_{n+1}, \ldots , x_{2n-1}], x_{i+1}, \ldots , x_{n}],
\end{align}
for $x_i \in \L$.  When we need to specify $n$, we say Filippov $n$-algebra.  Clearly  Filippov $2$-algebras are the same as ordinary Lie algebras.
In \cite{FHSW1}, 
Filippov $n$-algebras are referred to as {\em Lie algebras of the nth kind} (LAnKe's for short, see \cite{Fr}).

As defined in \cite{FHSW1}, the {\em free Filippov $n$-algebra} (or {\em free LAnKe}) on a set $X= \{ x_1,x_2, \dots,x_m\}$ is a vector space generated by all possible $n$-ary bracketings of elements of $X$, where the only
possible relations existing among these bracketings are consequences
of $n$-linearity of the bracketing, the antisymmetry of the bracketing  (\ref{antsym}) and the generalized Jacobi identity  (\ref{type1}).

The multilinear component of the free Filippov $n$-algebra on $X$ is the subspace spanned by $n$-bracketed permutations of $X$.  Every such $n$-bracketed permutation on $X$ has the same number of brackets, which depends on $m$, the number of generators.  Indeed, if the number of brackets is $k$ then   $m=(n-1)k +1$.

The object we studied in \cite{FHSW1} and continue to study in this paper is 
the representation of the symmetric group $\sg_{(n-1)k +1}$ on the
multilinear component of the free Filippov $n$-algebra on $(n-1)k +1$ generators. 
We denote this representation by $\rho_{n,k}$ and note that $\rho_{2,k}$ is the classical Lie representation $\lie_{k+1}$. It follows from the antisymmetry of the bracket that when $k=1$, $\rho_{n,k}$ is the sign representation $\sgn_n$, which is the same as the Specht module $S^{1^n}$. A result of Kraskiewicz and Weyman \cite{KW}  on $\lie_{m}$ gives  the decomposition of $\rho_{n,k}$ into irreducibles when $n=2$.  Recall that the {\it major index} of a standard Young tableau $T$ is the number of entries $j $ of $T$ for which  $j+1$ is in a lower row  of $T$ than $j$.

\begin{theorem}[Kraskiewicz and Weyman \cite{KW}] \label{liek}
 For each partition $\lambda \vdash m \ge 2$, the multiplicity of the Specht module $S^\lambda$ in $\lie_m$ is equal to the number  of standard Young tableaux of shape $\lambda$ and  major index congruent to $i \bmod m$, where $i$ is any fixed positive integer relatively prime to $m$.  
  \end{theorem}

Table~\ref{decomptable}  summarizes  the decompositions of $\rho_{n,k}$ into irreducibles that were previously established or will be established in this paper.
The partitions $\lambda$ in the table stand for Specht modules $S^\lambda$.  

\begin{table}[!t] \label{decomptable}

\begin{minipage}{\linewidth}

  \begin{center}\begin{tabular}{c|c|c|c|c}
  $n \backslash k$ & $1$  & $2$ &  $3$  & $4$ 
  \\[-2\medskipamount] & & & &  \\  \hline  & & &  \\[-2\medskipamount]      
$2$ & $1^2$ & $21$ & $31 \oplus 21^2$ & $41\oplus  32 \oplus   31^2 \oplus    2^21 \oplus 21^3$ \\   
$3$ & $ 1^3$   & $2^21$ & $3^21 \oplus 321^2$ & $4^21\oplus  432 \oplus   431^2 \oplus    42^21 \oplus 421^3   $ 
\\ & & & & $ \oplus  3^21^3 \oplus 32^3$ 
\\  
$4$ &  $1^4$  & $2^31$ & $3^31 \oplus 3^2 21^2$ & $4^31\oplus  4^232 \oplus   4^231^2 \oplus    4^2 2^21 \oplus 4^2  21^3   $ 
\\ & & & & $ \oplus  43^21^3 \oplus 432^3$ 
\\  \vdots & & & \\
$n$ & $1^n$ & $2^{n-1}1$ & $3^{n-1}1 \oplus 3^{n-2}21^2$ & $4^{n-1}1\oplus  4^{n-2}32 \oplus   4^{n-2}31^2 \oplus    4^{n-2} 2^21 \oplus 4^{n-2}  21^3   $ 
\\ & & & & $ \oplus  4^{n-3}3^21^3 \oplus 4^{n-3}32^3$ 

\end{tabular} \end{center}
\end{minipage}

\vspace{.1in}
\caption{Decomposition of $\rho_{n,k}$ for $k \le 4$.}
\label{knowtable}
\end{table}

In \cite{FHSW1}, the $k=2$ column of the table was established with the following result.
\begin{theorem}[\cite{FHSW1}]\label{k=2}  For all $n\ge 2$, the
 $\sg_{2n-1}$-module $\rho _{n,2}$ is isomorphic to the Specht module $S^{2^{n-1}1}$,
  whose dimension is the $n^{th}$ Catalan number $\frac{1}{n+1}{2n\choose n}$.  
\end{theorem}
In this paper we establish the $k=3$ column of the table with the following result.
\begin{theorem} \label{k=3} For all $n\ge 2$, as
 $\sg_{3n-2}$-modules,   $$\rho _{n,3} \cong S^{3^{n-1}1} \oplus S^{3^{n-2}21^2}.$$
 Consequently, by the hook length formula, $\dim \rho _{n,3} = \displaystyle \frac{4}{ \prod_{i=1}^3 (n+ i)} {3n \choose n,n,n}$.
\end{theorem}

We also establish the decomposition for $\rho_{n,4}$ given in the table.
 One can see from the table that for $k\in \{1,2,3,4\}$ and $n\ge k$,    increasing $n$ by 1 adds a row of length $k$ to the top of each Young diagram indexing the Specht module in the decomposition of $\rho_{n, k}$.
 More precisely, for $k \in \{1,2,3,4\}$,  we have 
\begin{equation}\label{introbetaeq} \rho_{n,k} \cong \beta_{n,k} \iff n \ge k,\end{equation} 
where $\beta_{n,k}$ is the $\sg_{k(n-1)+1}$-module whose decomposition into irreducibles is obtained by adding a row of length $k$ to the top of each Young diagram in the decomposition of $\rho_{n-1,k}$.  Here we set $\beta_{1,k} = \rho_{1,k} = S^1$.

The question of whether  (\ref{introbetaeq}) is true in general was raised in \cite{FHSW1}.  In this paper we obtain an affirmative answer to one direction with our main result, which is given in the following theorem,  and we present a conjecture which  refines the other direction.  Our main result indicates  that $\rho_{n,k}$ stabilizes in a certain sense when $n$ exceeds $k$ or is equal to $k$.

\begin{theorem}[Stabilization Theorem] \label{introstabth} Let $n \ge k \ge 1$.  Then   as $\sg_{k(n-1)+1}$-modules, 
\begin{equation}\label{introstabeq} \rho_{n,k} \cong \beta_{n,k}.\end{equation} \end{theorem}

We also prove the following theorem, which is used in the proof of the Stabilization Theorem and also gives some information on what happens above the $n=k$ diagonal.
\begin{theorem}  \label{decomposition_thm} Let $n,k \ge 1$.  Then as $\sg_{k(n-1)+1}$-modules,
\begin{equation} \label{gambet} \rho_{n,k} \cong   \beta_{n,k} \oplus \gamma_{n,k}, \end{equation}
for some  $\sg_{k(n-1)+1}$-module $\gamma_{n,k}$ all of whose irreducibles  have Young diagrams with at most $k-1$ columns.  
\end{theorem}

The following conjecture refines Conjecture~4.2 of  \cite{FHSW1}, which addresses the other direction of (\ref{introbetaeq}). 
\begin{conjecture} \label{lowconj} If $n \le j < k$ then there is an irreducible in $\gamma_{n,k}$ that has exactly $j$ columns.  \end{conjecture}

Note that the decompositions for $\rho_{n,2}$ and $\rho_{n,3}$, given in Theorems~\ref{k=2} and~\ref{k=3}, respectively,  are immediate consequences of the Stabilization Theorem.    The decomposition for $\rho_{n,4}$ given in Table~\ref{knowtable} also follows immediately from the Stabilization Theorem once the decomposition for the  entries $\rho_{2,4}$ and  $\rho_{3,4}$ are determined.  The former can be computed by applying Theorem~\ref{liek} and  the latter is computed in this paper by using Theorem~\ref{decomposition_thm}, the decompositions for $\rho_{2,4}$ and $\rho_{3,3}$, and a computer calculation for $\dim \rho_{3,4}$.

As a tool in proving  the Stabilization Theorem, we  initiate a study of two types of generalizations of the notion of Specht module in which Stanley's  $P$-partitions \cite[Sec. 3.15.1]{St1}, where $P$ is a tree,  assumes the role of ordinary partitions in the definition of Specht module. These tree Specht modules are the same as ordinary Specht modules  when $P$ is a path.

\begin{remark}  Theorems~\ref{k=3} and~\ref{decomposition_thm} were announced in our extended abstract  \cite{FHSWfpsac}  and in our  paper \cite{FHSW1}, with a reference to our current paper (which was in preparation at the time of publication of \cite{FHSW1}). Proofs of these theorems were posted in the first version of this paper  (arXiv:2402.19174v1).

The decomposition for $\rho_{n,4}$  given in Table~1 was proposed in   \cite[Proposition 4.7]{FHSW1} and the proof for $n=3$  was added  in the second version of this paper (arXiv:2402.19174v2) posted on March 19, 2024. 
The proof for general $n$ was added as a consequence of the $n=3$ case and Theorem~\ref{introstabth} in the third version (arXiv:2402.19174v3) posted on October 13, 2024.

Speculation on Theorem~\ref{introstabth} (and Theorem~\ref{rhocolth})  was raised in  \cite[Section 4]{FHSW1} and a proof  was announced  in a talk by one of the authors at the  June 2024 conference
``The Many Combinatorial
Legacies of Richard P. Stanley";  the slides are available at

 {\small https://www.math.harvard.edu/event/math-conference-honoring-richard-p-stanley/}.

\noindent Proofs of these theorems,  as well as the material on tree Specht modules, were added in the third version (arXiv:2402.19174v3) posted on  October 13, 2024.

Motivated by  \cite{FHSW1} and   an earlier version of  the current paper, Maliakas and Stergiopoulou \cite{MS} have recently posted  preprints (arXiv:2401.09405, arXiv:2410.06979) presenting  proofs of Theorems~\ref{k=3} and~\ref{rhocolth}, which appear to be very different from our approach.
\end{remark}

This paper is organized as follows. In Section~\ref{specsec}, after recalling a presentation for Specht modules appearing in \cite{Fu}, we obtain a result on the kernel of a certain linear operator on the induction product of Specht modules. In addition to this result playing an important role in subsequent sections,   a new presentation for Specht modules is obtained as a byproduct.     In Section~\ref{generalsec}, we prove Theorem~\ref{decomposition_thm} and we obtain  preliminary results on  the submodules of $\rho_{n,k}$ spanned by ``combs''  and  ``non-combs".
These results  are used in a proof of Theorem~\ref{k=3} (that doesn't rely on the Stabilization Theorem), which is given in Section~\ref{ncsec}. The proof of the  Stablilization Theorem is divided into two parts with Part 1 in Section~\ref{treequotsec} and Part 2 in Section~\ref{trspecsec}.  Our main tools in proving the Stabilization Theorem are the tree Specht modules of the first kind   introduced in Section~\ref{treequotsec} and the tree  Specht modules of the second kind  introduced in Section~\ref{trspecsec}.  In Section~\ref{fursec}, we establish the decomposition of $\rho_{n,4}$ given in Table~\ref{knowtable} and we discuss Conjecture~\ref{lowconj}.

\section{A linear operator on the induction product of Specht modules} \label{specsec} In this section, we obtain a result on a linear operator  that will play an important role in subsequent sections.  A more immediate application of this result  is a presentation for Specht modules that we have not seen in the literature before except for the special case of partitions with two columns \cite{BF}, where different techniques were used.

For ease of notation, throughout this paper, we will use the notation $A \bullet B$, for $A$ an $\sg_n$-module and $B$ an $\sg_m$-module, to denote the {\em induction product}
$$A \bullet B = (A \otimes B) \uparrow_{\sg_n \times \sg_m}^{\sg_{n+m}} ,$$ 
which is an associative operation that is distributive over direct sum.

\subsection{Preliminaries on Specht modules} \label{prelimsubsec} The Specht modules $S^\lambda$, where $\lambda$ is a partition of $n$, give a complete set of  irreducible 
representations of the symmetric group $\sg_n$ over a field  of characteristic $0$, say $\mathbb C $.  They can be constructed as  
subspaces of the regular representation $\mathbb C \sg_n$  or  as presentations given in terms of generators and relations, 
known as Garnir relations.   We will need the latter type of construction in this paper.

Let $\lambda= (\lambda_1 \geq\dots \geq \lambda_l)$ be a partition of $n$.  A {\em Young tableau} of shape $\lambda$ is a filling of the Young diagram associated with $\lambda$ with distinct entries from the set $[n]:=\{1,2,\dots,n\}$.  Let $\mathcal T_\lambda$ be the set of  Young tableaux of shape
$\lambda$.

To construct the Specht module as a presentation, one can use column tabloids and Garnir relations. Let $M^\lambda$ be the vector space (over $\C$) generated by  
$\mathcal T_\lambda$ subject only to column relations, which are of
the form $t+s$, where $s\in \mathcal T_\lambda$ is obtained from $t\in \mathcal T_\lambda$ by switching two
entries in the same column.  
Given $t \in \mathcal T_\lambda$, let $\bar t$ denote the coset of $t$
in $M^\lambda$. These cosets, which are called {\it column tabloids}, generate $M^{\lambda}$.
A Young tableau is {\em column strict} if the entries of each of its columns increase from top to bottom.  Clearly,
$\{\bar t  :  t \mbox{ is a column strict Young tableau of shape } \lambda\}$ is a basis for $M^\lambda$.

We will use a presentation of $S^\lambda$ discussed in Fulton \cite[page 102 (after Ex.~15)]{Fu}. Suppose $\lambda$ has $m$ columns, labeled from left to right. The generators are the column tabloids $\bar t$, where $t \in \mathcal T_\lambda$.
The Garnir relations are of the form\footnote{In our previous papers \cite{FHSWfpsac, FHSW1, FHW}, we  denote $g^t_{c}$ by  $g^t_{c,1}$.}
\begin{equation} \label{gareq} g^t_{c}:=
\bar t-\sum_s \bar s,\end{equation} where  $c \in \mathcal [m-1]$, $t \in  \mathcal T_\lambda$, and the sum is over all $s\in \mathcal
T_\lambda$ obtained from $t\in \mathcal T_\lambda$ by exchanging any
entry of  column $c$  with the top entry of the next
column. 

\vspace{.1in}\noindent {\bf Example:} For
$$t= \ytableausetup
{mathmode, boxsize=1em}
\begin{ytableau}
\scriptstyle 1 & \scriptstyle 5 & \scriptstyle 7 \\
 \scriptstyle 2 &  \scriptstyle 6  \\
 \scriptstyle 3 \\
\scriptstyle 4
\end{ytableau}\,\,,$$
we have
$$ g^t_{1} =
\ytableausetup
{mathmode, boxsize=1em}
\overline{\begin{ytableau}
\scriptstyle 1 & \scriptstyle 5 & \scriptstyle 7 \\
 \scriptstyle 2 &  \scriptstyle 6  \\
 \scriptstyle 3 \\
\scriptstyle 4
\end{ytableau}}
\,\,- \,\,
\ytableausetup
{mathmode, boxsize=1em}
\overline{\begin{ytableau}
\scriptstyle 5 & \scriptstyle 1  & \scriptstyle 7 \\
 \scriptstyle 2 &  \scriptstyle 6  \\
 \scriptstyle 3 \\
\scriptstyle 4
\end{ytableau}}
\,\,- \,\,
\ytableausetup
{mathmode, boxsize=1em}
\overline{\begin{ytableau}
\scriptstyle 1 & \scriptstyle 2  & \scriptstyle 7\\
 \scriptstyle 5 &  \scriptstyle 6  \\
 \scriptstyle 3 \\
\scriptstyle 4
\end{ytableau}}
\,\,- \,\,
\ytableausetup
{mathmode, boxsize=1em}
\overline{\begin{ytableau}
\scriptstyle 1 & \scriptstyle 3 & \scriptstyle 7 \\
 \scriptstyle 2 &  \scriptstyle 6  \\
 \scriptstyle 5 \\
\scriptstyle 4
\end{ytableau}}
\,\,- \,\,
\ytableausetup
{mathmode, boxsize=1em}
\overline{\begin{ytableau}
\scriptstyle 1 & \scriptstyle 4 & \scriptstyle 7 \\
 \scriptstyle 2 &  \scriptstyle 6  \\
 \scriptstyle 3 \\
\scriptstyle 5
\end{ytableau}} \,\,.
$$

Permutations  $\sigma$ in $\sg_n$ act on a tableau $t$ by replacing  each entry of $t$ with its image under $\sigma$.  Note that this action respects the column relations and the Garnir relations.
\begin{proposition}[\cite{Fu}] \label{presentprop} For all $\lambda \vdash n$, the Specht module
$S^\lambda$ is generated by the column tabloids  subject only to the Garnir relations
$g^t_{c}$, where $t \in
\mathcal T_\lambda$ and  $c \in [m-1]$, where $m$ is the number of columns of $\lambda$.
\end{proposition}

We say that a Specht module $S^\lambda$, $\lambda \vdash n$,  is contained in an $\sg_n$-module $V$ if  $S^\lambda$ is isomorphic to a submodule of $V$, and we say that its multiplicity in $V$ is $c_\lambda$ if $c_\lambda$ is the largest integer  such that the direct sum of $c_\lambda$ copies of $S^\lambda$ is isomorphic to a submodule of $V$.

Given a partition $\lambda$, we use the notation $\lambda^*$ to denote the {\em conjugate partition}.  

\subsection{The linear operator}
Given partitions $\lambda_1$ and $\lambda_2$ viewed as Young diagrams, we will say that 
$\lambda_1$ and $\lambda_2$ are {\em compatible} if the  length of the last column of  $\lambda_1$ is greater than or equal to the length of the first column of $\lambda_2$.   In that case, we  let $\lambda_1 \conc \lambda_2$ denote the {\em concatenation} of $\lambda_1$ and $\lambda_2$, i.e., the Young diagram obtained by placing the Young diagram $\lambda_2$ just to the right of the Young diagram $\lambda_1$ and aligning the first rows of both. Note that $\lambda_1\conc \lambda_2$  is the partition obtained by componentwise addition of $\lambda_1$ and $\lambda_2$, and  that $\lambda_1 \conc \lambda_2 \vdash |\lambda_1| + |\lambda_2|$.  Note also that $(\lambda_1\conc \lambda_2)/ \lambda_1 = \lambda_2$.

\vspace{.1in}\noindent {\bf Example:} Let
$$\lambda_1= \ytableausetup
{mathmode, boxsize=1em}
\begin{ytableau}
\phantom{1} &\phantom{1} &\phantom{1} \\ \phantom{1} & \phantom{1} & \phantom{1} \\ \phantom{1} & \phantom{1} & \phantom{1} \\ \phantom{1} & \phantom{1} \\ \phantom{1} & \phantom{1} \\ \phantom{1} \end{ytableau}.
\quad \qquad \lambda_2= \ytableausetup
{mathmode, boxsize=1em}
\begin{ytableau}
\phantom{1} &\phantom{1} &\phantom{1} \\ \phantom{1} & \phantom{1}
\end{ytableau}
\quad \qquad \lambda_3= \ytableausetup
{mathmode, boxsize=1em}
\begin{ytableau}
\phantom{1} &\phantom{1} &\phantom{1} \\ \phantom{1} & \phantom{1} \\ \phantom{1} & \phantom{1} \\\phantom{1} & \phantom{1}
\end{ytableau}\,\,.$$
Then $\lambda_1$ and $\lambda_2$ are compatible but $\lambda_1$ and $\lambda_3$ are not, and so $\lambda_1$ and $\lambda_2$ can be concatenated  but  $\lambda_1$ and $\lambda_3$ cannot be concatenated.  The concatenation is $$\lambda_1 \conc \lambda_2 =  \ytableausetup
{mathmode, boxsize=1em}
\begin{ytableau}
\phantom{1} &\phantom{1} &\phantom{1} &\phantom{1} &\phantom{1} &\phantom{1}
\\ \phantom{1} & \phantom{1} & \phantom{1}& \phantom{1} & \phantom{1} \\ \phantom{1} & \phantom{1} & \phantom{1} \\ \phantom{1} & \phantom{1} \\ \phantom{1} & \phantom{1} \\ \phantom{1} \end{ytableau} \,\,.$$

\begin{lemma} \label{indlem} Let  $\lambda_1$ and $\lambda_2$ be partitions of $n_1$ and $n_2$, respectively, and let $d$ be any integer that is greater than or equal to the number of columns of $\lambda_1$ viewed as a Young diagram. Define the linear operator $ \varphi_d = \varphi^{\lambda_1,\lambda_2}_d$ on the induction product $$ X := S^{\lambda_1} \bullet S^{\lambda_2} $$
by
$$ \varphi_d =  dn_2 I - \sum_{\substack{i \in [n_1] \\ j \in [n_1+n_2] \setminus  [n_1]} } (i,j),$$
where the action of $(i,j)$ is on the right.  Then  the eigenvalues of $\varphi_d$ are nonnegative and
$$\ker( \varphi_d) = \begin{cases} S^{\lambda_1\conc\lambda_2} &\mbox{ if  $\lambda_1$ and $\lambda_2$ are compatible and $\lambda_1$ has $d$ columns} 
  \\ 0 &\mbox{otherwise.}  \end{cases}$$
\end{lemma}

\begin{proof} Rewrite $\varphi_d $ as 
\begin{equation} \label{phieq}  \varphi_d = dn_2 I - \sum_{(i,j) \in \sg_{n_1+n_2} } (i,j) + \sum_{(i,j) \in \sg_{n_1}} (i,j) + \sum_{(i,j) \in \sg_{[n_1+n_2]\setminus [n_1]}} (i,j).
\end{equation}
For any partition $\lambda$,  let $$C(\lambda) = \sum_{(r,s) \in \lambda} c_\lambda(r,s),$$ where $c_\lambda(r,s)= s-r$ for the cell in row $r$ and column $s$ of $\lambda$. 
A well-known result from the representation theory of the symmetric groups (see \cite[Section I7, Example 7]{Ma}) states that the operator $\sum_{(a,b) \in \sg_n} (a,b)$ acts as multiplication by the scalar $C(\lambda)$ on the irreducible $S^\lambda$.  
Applying  this result to the summands of (\ref{phieq}),  we have  that $\varphi_d$ acts as the scalar 
 $$a_\lambda :=dn_2 - C(\lambda) +C(\lambda_1) + C(\lambda_2)
 $$ 
 on every instance of the irreducible $S^\lambda$, $\lambda \vdash n_1+n_2$, contained in $X$.

The irreducibles $S^\lambda$ contained in $X$ and their multiplicities are given by the Littlewood-Richardson rule, see  \cite{Fu, Ma, St2}.  In particular, any such $\lambda$ must contain $\lambda_1$ and the multiplicity of $S^\lambda$ in $X$ is the number of Littlewood-Richardson fillings of the skew shape $\lambda/\lambda_1$ of type $\lambda_2$.  It follows from the containment of $\lambda_1$ in $\lambda$ that 
$$a_\lambda =dn_2+C(\lambda_2)-\sum_{(r,s)\in \lambda/\lambda_1} c_\lambda(r,s).$$
Let $l(\lambda_2)$ be the length of $\lambda_2$ and consider the partition $d^{l(\lambda_2)}$, which is compatible with  $\lambda_2$, and 
let $\nu = d^{l(\lambda_2)} \conc \lambda_2$. Then
$$\sum_{(r,s) \in \nu/d^{l(\lambda_2)}} c_\nu(r,s) = \sum_{(r,s) \in \lambda_2} (c_{\lambda_2}(r,s)+d) = dn_2+C(\lambda_2).$$ Hence 
\begin{equation} \label{eigeq} a_\lambda =\sum_{(r,s) \in \nu/d^{l(\lambda_2)}} c_\nu(r,s)- \sum_{(r,s)\in \lambda/\lambda_1} c_\lambda(r,s).\end{equation}

 We will show that the eigenvalue $a_\lambda$ is nonnegative by constructing a bijection $\psi$ from the set of cells of $\lambda/\lambda_1$ to the set of cells of $\nu/d^{l(\lambda_2)}$ that satisfies 
\begin{equation} \label{decconteq} c_\lambda(r,s) \le c_\nu(\psi(r,s))\end{equation}
for all $(r,s) \in  \lambda/\lambda_1$. First fix a Littlewood-Richardson filling of $\lambda/\lambda_1$ of type $\lambda_2$.
Since the filling is column strict, for each  $i \in [l(\lambda_2)]$, the entry $i$ of the filling occurs in $(\lambda_2)_i$ distinct columns of $\lambda/\lambda_1$, where $(\lambda)_i$ denotes the $i$th  part of a partition $\lambda$. Let $(r,s)$ be a cell of  $ \lambda/\lambda_1$. Suppose that $(r,s)$ contains the $j$th occurrence (from the left) of entry $i$ of the filling and let $\psi(r,s)=(i, d+j)$.    It is not difficult to see that $\psi$ is a well defined bijection  from the set of cells of $\lambda/\lambda_1$ to the set of cells of $\nu/d^{l(\lambda_2)}$. Next we show that (\ref{decconteq}) holds.

We use the property of  Littlewood-Richardson fillings  that the rows contain no entries larger than the row index.  This yields $r \ge i$.  We claim that $s \le d+j$. Indeed, if $s\le \lambda_1$ then $s\le d < d+j$.  On the other hand, let $s > \lambda_1$.    It follows from column strictness of  Littlewood-Richardson fillings and the property above that  $r=i$ and $i$ is the entry of each cell $(r,m)$, $m= \lambda_1+1,\dots, s$.   This implies that $j \ge s-\lambda_1$, which implies that $d+j \ge d+s-\lambda_1 \ge s$.  Therefore the claim holds in both cases, which implies   $$c_\lambda(r,s) = s-r \le d+j - i = c_\nu(i,d+j) = c_\nu(\psi(r,s)).$$ Thus the inequality
(\ref{decconteq}) holds for all $(r,s)$ in $\lambda/\lambda_1$.  Looking back at the proof of the inequality, observe that equality in (\ref{decconteq}) holds if and only if $s > \lambda_1$, $j= s-\lambda_1$, and $(\lambda_1)_i=d$, for all $i=1,\dots, l(\lambda_2)$, which means that equality holds if and only if $\lambda = \lambda_1 \conc \lambda_2 $ and $\lambda_1$ has $d$ columns.  

Now by (\ref{eigeq}) and (\ref{decconteq}),
$$ a_\lambda =\sum_{(r,s) \in \lambda/\lambda_1} (c_\nu(\psi(r,s))- c_\lambda(r,s) )\ge 0.$$
The above conditions on equality holding in (\ref{decconteq}) allows us to conclude that the eigenvalue $ a_\lambda $ is $0$ if and only if $\lambda = \lambda_1 \conc \lambda_2 $ and $\lambda_1$ has $d$ columns. 
\end{proof}

 Lemma~\ref{indlem} can be used to give a  presentation for Specht modules different from Fulton's presentation given in Proposition~\ref{presentprop}.  Indeed, let $\lambda$ be a partition of $n$ with $m$ columns. For $t \in \mathcal T_\lambda$ and $c \in \{2,\dots,m\} $, define the {\em new} Garnir relation 
\begin{equation} \label{newgareq} \tilde g_c^t := (c-1) l_c \, \bar t - \sum_s \bar s, \end{equation}
where $l_c$ is the length of column $c$ and the sum is over all tableaux $s$ that can be obtained from $t$ by switching any element of column $c$ with any element of any previous column.   Now let
 $$\tilde S^\lambda := M^\lambda / \langle \tilde g^t_c : t \in \mathcal T_\lambda, \, 2 \le c \le m \rangle .$$ 
 
 \begin{theorem} \label{newpresth} For all $\lambda \vdash n$,
 $$\tilde S^\lambda \cong S^\lambda.$$
 \end{theorem}
 
 \begin{proof} The proof  is by induction on the number $m$ of columns of $\lambda$. The result is obvious in the base case $m=1$.
Assume that $m \ge 2$ and that the result is true for all partitions with  fewer columns. Let $\lambda^\prime$ be  the partition obtained from $\lambda$ by removing its last column. Since $M^\lambda = \sgn_{l_{c_1}} \bullet \,\, \cdots \,\, \bullet \sgn_{l_{c_m}} = M^{\lambda^\prime} \bullet \sgn_{l_{c_m}} $, we have $$\tilde S^{\lambda} \cong  ( \tilde S^{\lambda^\prime}  \bullet \sgn_{l_m}) / \langle \tilde g_m^t : t \in \mathcal T_{\lambda}\rangle.
$$
By the induction hypothesis, $\tilde S^{\lambda^\prime} \cong S^{\lambda^\prime}$.  
Since $\lambda^\prime$ is clearly compatible with the partition $1^{l_m}$  and $\lambda = \lambda^\prime \conc 1^{l_m} $,   by Lemma~\ref{indlem}, 
$$\tilde S^{\lambda} \cong   ( S^{\lambda^\prime} \bullet \sgn_{l_m} ) /  \im \varphi_{m-1} \cong S^{\lambda},$$ since $\langle \tilde g_m^t : t \in \mathcal T_{\lambda} \rangle= \im \varphi_{m-1}$.
 \end{proof}

 \begin{remark} \rm{Theorem~\ref{newpresth} in the special case that $\lambda$ has just two columns was obtained in  \cite{BF} using different techniques.} \end{remark}
 
The following  generalization of Lemma~\ref{indlem} is an easy consequence of Lemma~\ref{indlem}.

\begin{lemma} \label{indlem2} For $i=1,2$, let $Y_i$ be an  $\sg_{n_i}$-module. Let $d$ be an integer that is greater than or equal to the number of columns of every irreducible in $Y_1$.  Define the linear operator $ \varphi_d = \varphi^{Y_1,Y_2}_d$ on the induction product $$ X := Y_1 \bullet Y_2 $$  by
$$ \varphi_d :=  dn_2 I - \sum_{\substack{i \in [n_1] \\ j \in [n_1+n_2] \setminus  [n_1]} } (i,j),$$
where the action of $(i,j)$ is on the right.  Then $S^\lambda$ is an irreducible in $\ker \varphi_d$ if and only if  $\lambda$ has the form $\lambda_1\conc\lambda_2$, where 
\begin{enumerate}
\item $S^{\lambda_i}$ is an irreducible of $Y_i$, $i=1,2$ 
\item $\lambda_1$ has $d$ columns and is compatible with $\lambda_2$. 
\end{enumerate}
\end{lemma}

\begin{proof} Let $S^\lambda$ be an irreducible of $X$. Then $S^\lambda$ is in the kernel of $\varphi_d$ if and only if $S^\lambda$  is in the kernel of the restriction of $\varphi_d$ to the induction product $S^{\lambda_1} \bullet S^{\lambda_2}$ for some irreducible $S^{\lambda_1}$ of $Y_1$, and $S^{\lambda_2}$ of $Y_2$. By Lemma~\ref{indlem}, the kernel of this restriction is $S^{\lambda_1\conc \lambda_2}$ if $\lambda_1$ has $d$ columns and is compatible with $\lambda_2$.  Otherwise the kernel is $0$. Hence $\lambda$ has the desired form if and only if $S^\lambda$ is in the
 kernel of $\varphi_d$.
\end{proof}

\section{Decomposition  results} \label{generalsec}

\subsection{Preliminary results on free Filippov Algebras} \label{Filsubsec}

In Appendix 1 of \cite{DI}, a proof that the   generalized Jacobi relations (\ref{type1}) are equivalent to the relations  
\begin{align} \label{filipeq} & [[x_1, x_2, \ldots,  x_n], y_{1},\ldots , y_{n-1}] \\
\nonumber = &\sum_{i=1}^{n} [[x_1, x_{2}, \ldots , x_{i-1},y_{1}, x_{i+1}, \ldots , x_{n}],x_{i},y_{2},\ldots , y_{n-1}]
\end{align}
 is given.
 This gives an alternative useful presentation of $\rho_{n,k}$ for all $n,k$.

An $n$-bracketed word on a set $A$
is called an  $n$-{\it comb} (or just  {\it comb}) if it has  the form
\begin{equation} \label{combeq} [[\cdots [[a_0,a_{1,1},\dots, a_{1,n-1}],a_{2,1},\dots, a_{2,n-1}] ,\dots ], a_{k,1},\dots, a_{k,n-1}],
\end{equation}
where $a_0,a_{i,j} \in A$.

\begin{proposition} \label{coneprop} Let $Lie_{n,k}(X)$ be the component of the free Filippov $n$-algebra on a set $X$ spanned by $n$-bracketed words with $k$ brackets.  Then 
 the combs in $Lie_{n,k}(X)$ span $Lie_{n,k}(X)$.
\end{proposition}

\begin{proof} Let
$w$ be a bracketed word in the free Filippov $n$-algebra. We will prove by induction on  the number  $b(w)$ of brackets of $w$ that $w$ is a linear combination of combs.  If $b(w)=0,1$ then $w$ is clearly a comb.  Now let $b(w) \ge 2$ and
$w=[w_1,w_2,\dots,w_n]$.
Define {\it degree} $d(w) :=   \max_{i\in [n]} b(w_i)$.  

We claim that if $w$  has the highest possible degree $b(w)-1$ then $w$ is a linear combination of combs.  Indeed, in this case  all but one $w_i$ is a single letter.  Using antisymmetry of the bracket, we can assume without loss of generality that $b(w_1) = b(w)-1$ and $b(w_i) = 0$ for $i=2,\dots,n$.  By induction $w_1$ is a linear combination of combs, which implies by multilinearity of the bracket that $w$ is a linear combination of combs. 

Now let $d(w)< b(w)-1$. We only need to show that  $w$ is equal to a linear combination of bracketed words of higher degree.  Clearly there are at least two $w_i$ for which $b(w_i) >0$. .  By antisymmetry of the bracket, without loss of generality, we can assume that $0 < b(w_1) \le b(w_2) $ and that $d(w) = b(w_2)$.  Let $w_1 = [u_1,\dots,u_n]$.  Using the alternative generalized Jacobi relation given in (\ref{filipeq}), we have
$$w= \sum_{i=1}^n [[u_1,u_2, \dots, u_{i-1}, w_2, u_{i+1}, \dots, u_n], u_i, w_3,w_4,\dots, w_n].$$
Note that the degree of each term  is greater than or equal to $b(w_2) +1= d(w)+1$.  Thus $w$ is a sum of bracketed words of degree greater than $d(w)$.  
\end{proof}

Let $V_{n,k}$ be the multilinear component of the vector space generated by $n$-combs  on $[k(n-1)+1]$  subject only to the anticommuting relations, but not the generalized Jacobi relations.  Clearly $V_{n,k}$ is  invariant under the action of $\sg_{k(n-1)+1}$ on the multilinear component of the full vector space of bracketed permutations on $[k(n-1)+1]$ subject only to the anticommuting relations.  By Proposition~\ref{coneprop}, $\rho_{n,k}$ is a quotient (or submodule) of $V_{n,k}$. Note that as $\sg_{k(n-1)+1}$-modules,
 \begin{equation} V_{n,k} \cong   \sgn_{n}\bullet \sgn_{n-1}^{\bullet (k-1)}. 
 \end{equation} 
As a consequence, by Pieri's rule, we have the following useful result.

\begin{proposition} \label{lekprop} Let $\lambda$ be a partition of $k(n-1) +1$.  If $S^\lambda$ is  in $\rho_{n,k}$ then the number of columns of $\lambda$, viewed as a Young diagram, is at most $k$.
\end{proposition}

\subsection{Proof of Theorem~\ref{decomposition_thm}} Let $\hat \rho_{n,k}$ be the subspace of the free Filippov $n$-algebra on ${[k(n-2)+2]}$ spanned by bracketed words of type $1^{k(n-2)+1}k$, where the type of a bracketed word is the weak composition $(\mu_1,\mu_2,\dots,\mu_{k(n-2)+2})$ with $\mu_i$ equaling the number of $i$'s in the bracketed word.  This means that the letters $1,2,\dots, k(n-2)+1$ each occur once in the bracketed words generating $\hat \rho_{n,k}$  and the letter $b:=k(n-2)+2$ occurs $k$ times. The symmetric group $\sg_{k(n-2)+1}$ acts on $\hat \rho_{n,k}$ by replacing each letter $1,\dots, k(n-2)+1$  in a bracketed word with its image under a permutation in $\sg_{k(n-2)+1}$, leaving $b$ fixed.

We also need  an extended definition of  Specht modules that allows for repeated entries.  We define a {\it tableau} to  be a filling of a shape $\lambda$ with positive integer entries. The
{\it type} of a tableau $t$  is the weak composition whose $i$th entry is the number of $i$’s
appearing in $t$.  For each partition $\lambda$ of $k(n-1)+1$, let   $\hat {\mathcal T}_\lambda$ be the set of tableaux of shape $\lambda$ and type $1^{k(n-2)+1} k$.  The column relations, column tabloids, and Garnir relations on $\hat {\mathcal T}_\lambda$ are the same as on $\mathcal T_\lambda$ given in Section~\ref{prelimsubsec}.  Now let $\hat S^\lambda$ be the vector space generated by column tabloids $\bar t$, where $t\in \hat {\mathcal T}_{\lambda} $, subject to the Garnir relations $g^t_c$.   The symmetric group $\sg_{k(n-2)+1}$ acts on $\bar t$ by replacing each entry $1,\dots, k(n-2)+1$  of $t$ with its image under a permutation in $\sg_{k(n-2)+1}$, leaving $b$ fixed.

 In \cite[Proposition~43]{DI}, the idea of obtaining Filippov $(n-1)$-algebras from Filippov $n$-algebras is discussed.  We use this idea in the proof of the following result.
 
 \begin{lemma} \label{hatlem2} Let $n\ge 2$ and $k \ge 1$. 
  Then 
\begin{equation} \label{rho_eqn} \hat \rho_{n,k} \cong_{\sg_{k(n-2)+1}} \rho_{n-1,k}, \end{equation}
 and for all $\lambda \vdash k(n-1) +1$ with $k$ columns,
\begin{equation} \label{specht_eqn} \hat S^{\lambda}   \cong_{\sg_{k(n-2)+1}}  S^{\lambda^{-}} ,\end{equation}
where if $\lambda$ is the partition $ (\lambda_1 \ge \lambda_2 \ge \dots \ge  \lambda_\ell)$ then $ \lambda^{-}$ is the partition $(\lambda_2  \ge \dots \ge  \lambda_\ell)$.
\end{lemma}

\begin{proof}[Proof of (\ref{rho_eqn})]    The  $n$-ary bracketed words of type $1^{k(n-2) +1} k$ have $k$ occurences of the letter $b := k(n-1)+2$.  Since there are $k$ brackets, by antisymmetry of the brackets, each bracket of a nonvanishing $n$-ary bracketed word contains exactly one $b$ and the $b$ can be placed at the right end of each bracket.  For example, 
$$[ [ [3,1,b],[5,2,b],b ] ,[6,4,b],b]$$
 is such a $3$-bracketed word of type $1^{5(1)+1}5$.  These bracketed words, which we refer to as {\it $b$-bracketings}, generate $\hat \rho_{n,k}$.  
The antisymmetry relations reduce to   
$$[x_1,\dots,x_{n-1},b]  = \sgn(\sigma) [x_{\sigma(1)},\dots,x_{\sigma(n-1)},b],$$ where $\sigma \in \sg_{n-1}$,
for the $b$-bracketings, and the generalized Jacobi relations  reduce to
\begin{align*}
 &[[x_1, x_2, \ldots,  x_{n-1},b], x_{n+1},\ldots , x_{2n-2},b]
\\ \non&=\sum_{i=1}^{n-1} [x_1, x_{2}, \ldots , x_{i-1},[ x_{i}, x_{n+1}, \ldots , x_{2n-2},b], x_{i+1}, \ldots , x_{n-1},b]
\end{align*}
for the $b$-bracketings, since the $i=n$ term of the summation in (\ref{type1}) is 
$$[x_1,\dots,x_{n-1},[b,x_{n+1},\dots,x_{2n-2},b]],$$ which by antisymmetry is $0$.  These relations imply all relations for $\hat \rho_{n,k}$.  Indeed, one can check that any other relation on these generators that comes from  (\ref{type1}) is trivial.  For instance,  by antisymmetry of the bracket,
(\ref{type1}) with $x_{n+1} =x_{2n-1}= b$ yields the relation $0=0$ and (\ref{type1}) with $x_1 =x_n= b$ yields the relation
$0 = [[x_{n+1},\dots, x_{2n-1}, b], x_2,\dots, x_{n-1},b] -  [[x_{n+1},\dots, x_{2n-1}, b], x_2,\dots, x_{n-1},b]$.

Now let $\varphi: \hat \rho_{n,k} \to \rho_{n-1,k}$ be the linear map defined  by letting $\varphi(w)$ be the $(n-1)$-ary bracketed word obtained from the $b$-bracketing $w$ by removing all the  $b$'s.  For example,
$$\varphi([ [ [3,1,b],[5,2,b],b ] ,[6,4,b],b]) =  [ [ [3,1],[5,2] ] ,[6,4]].$$ We can see that  $\varphi$ respects the antisymmetry relations 
and the generalized Jacobi relations and is therefore a well defined linear map.  Indeed, $\varphi$ takes
$$[x_1,\dots,x_{n-1},b]  - \sgn(\sigma) [x_{\sigma(1)},\dots,x_{\sigma(n-1)},b]$$ to $$ [x_1,\dots,x_{n-1}]  - \sgn(\sigma) [x_{\sigma(1)},\dots,x_{\sigma(n-1)}] $$
for all $\sigma \in \sg_{n-1}$ and $\varphi$ takes 

\begin{align*}
 &[[x_1, x_2, \ldots,  x_{n-1},b], x_{n+1},\ldots , x_{2n-2},b]
\\ \non&-\sum_{i=1}^{n-1} [x_1, x_{2}, \ldots , x_{i-1},[ x_{i}, x_{n+1}, \ldots , x_{2n-2},b], x_{i+1}, \ldots , x_{n-1},b]
\end{align*}

to

\begin{align*}
 &[[x_1, x_2, \ldots,  x_{n-1}], x_{n+1},\ldots , x_{2n-2}]
\\ \non&-\sum_{i=1}^{n-1} [x_1, x_{2}, \ldots , x_{i-1},[ x_{i}, x_{n+1}, \ldots , x_{2n-2}], x_{i+1}, \ldots , x_{n-1}].
\end{align*}
Note that $\varphi$ can easily be inverted.  Hence $\varphi$ is a vector space isomorphism.  It is clear that $\varphi$ respects the action of $\sg_{k(n-2)+1}$; so it is an $\sg_{k(n-2)+1}$-module isomorphism.
\end{proof}

\begin{proof}[Proof of (\ref{specht_eqn})]  The proof is similar to the proof of (\ref{rho_eqn}).  The column tabloids $\bar t$ of shape $\lambda$ and type $1^{k(n-2)+1}k$ have $k$ occurrences of $b := k(n-2)+2$.  By antisymmetry of the columns, since $\lambda$ has $k$ columns, each column of $t$ must have exactly one $b$ and the $b$ can be placed at the bottom of the column.   We call such a $t$ a $b$-tableau.  Let $\mathcal T^b_{\lambda} $ be the subset of $\hat {\mathcal T}_\lambda$ consisting of $b$-tableaux of shape $\lambda$. The column relations on $b$-tableaux are of the form $t+s$, where $s$ is obtained from $t$ by exchanging two non-$b$ entries in the same column. The column tabloids $\bar t$, where $t \in \mathcal T^b_{\lambda}$,  generate   $\hat S^{\lambda}$.  The  Garnir relations for these generators reduce to $g^t_c = \bar t- \sum_{s} \bar s$, where $t \in \mathcal T^b_{\lambda}$ and   the sum is over all $s$ obtained from $t$ by exchanging any  element of column $c$ other than the bottom element $b$  with the top element of the next column (provided the next column has length greater than 1).

Given   a $b$-tableau $t$ of shape $\lambda$, let $t^-$ be the tableau obtained from $t$ by removing the $b$ from the bottom 
of each column. Clearly $t^-$ is a  tableau of shape $\lambda^{-}$. 
Now let $\psi: \hat S^{\lambda} \to 
S^{\lambda^{-}}$ be the linear map defined on the generating set $\{\bar t\mid t \in \mathcal T^b_\lambda\}$ by letting
 $\psi(\bar t) = \overline{t^-}$.  To see that this is a well-defined $\sg_{k(n-2)+1}$-module homomorphism, one  can easily verify that the column relations map to column relations,  and the Garnir relation $g^t_c $ maps to  the Garnir relation $g^{t^-}_c$,  for all $t \in  \mathcal T^b_\lambda$.  The inverse map  is easily obtained by adding a $b$ to the bottom of each column of a tableau of shape $\lambda^{-}$.   Again since relations map to relations under this map, the inverse is well-defined.  Hence $\psi $ is an isomorphism.  Alternatively, one can see that $\psi$ is an isomorphism, by noting that $\psi$ maps 
basis elements of   $\hat S^{\lambda}$  to basis elements of $S^{\lambda^{-}}$.  Indeed,  the column tabloids $\bar t$, where $t$ is a semistandard $b$-tableau of shape $\lambda$, 
 form a basis for $\hat S^{\lambda}$ and   the column tabloids $\bar t$, where $t$ is a standard tableau of shape $\lambda^-$ form a basis for $S^{\lambda^-}$\end{proof}
 
We need another lemma, which combined with the previous lemma yields  Theorem~\ref{decomposition_thm}.

\begin{lemma}  \label{hatlem1} As $\sg_{k(n-2)+1}$-modules
$$\hat \rho_{n,k} \cong \bigoplus_{\substack{ \lambda \vdash k(n-1) +1\\  \lambda \, \mbox{\scriptsize has $k$ columns }  } } c_\lambda \hat S^{\lambda},$$
where $c_\lambda$ is the multiplicity of the irreducible $\sg_{k(n-1) +1}$-module $S^{\lambda}$ in $\rho_{n,k}$.
\end{lemma}

\begin{proof}
Let $B:=\{k(n-2) + 2, k(n-2) + 3,\dots, k(n-1) + 1\}$ and let $\alpha := \sum_{\sigma \in \sg_B} \sigma$.   One can  see that 
$$\alpha ( \rho_{n,k} \downarrow^{\sg_{k(n-1)+1}}_{\sg_{k(n-2)+1}\times \sg_B} )\cong \hat \rho_{n,k} \otimes 1_{\sg_B}$$
and 
$$\alpha  (S^\lambda \downarrow^{\sg_{k(n-1)+1}}_{\sg_{k(n-2)+1} \times \sg_B }) \cong \hat S^\lambda \otimes 1_{\sg_B},
$$
where $1_{\sg_B}$ denotes the trivial representation of $\sg_B$.
Now take the restriction of  both sides of the decomposition 
$$\rho_{n,k} \cong \bigoplus_{\lambda \vdash k(n-1)+1} c_\lambda S^\lambda,$$
from $\sg_{k(n-1)+1}$ to  $\sg_{k(n-2)+1} \times \sg_B$ 
and then take the image under the action of $\alpha$ to get the $\sg_{k(n-2)+1}$-module isomorphism $$\hat \rho_{n,k} \cong \bigoplus_{\lambda \vdash k(n-1)+1} c_\lambda \hat S^\lambda.$$
By Proposition~\ref{lekprop}, $c_\lambda =0$ if $\lambda$ has more than $ k$ columns.  On the other hand $\hat S^\lambda = 0$ if $\lambda$ has fewer than  $k$ columns.  Indeed, if $t\in \hat {\mathcal T}_{\lambda} $ then    entry $b:=k(n-2)+2$ is repeated $k$ times. If $\lambda$ has fewer than  $k $ columns then entry $b$ will appear more than once  in  some column of $t$.  By  antisymmetry of the columns,  the column tabloid $\bar t$ is  equal to $0$. Thus we can limit the summation to $\lambda$ with  exactly $k$ columns.
 \end{proof}

We can now easily prove Theorem~\ref{decomposition_thm}, which we restate here.  First recall that for  $n \ge 2$ and $k \ge 1$, we defined  $\beta_{n,k}$ to be the $\sg_{k(n-1)+1}$-module whose decomposition into irreducibles is obtained from the $\sg_{k(n-2)+1}$-module $\rho_{n-1,k}$ by adding a row of length $k$ to the top of each Young diagram in the decomposition of $\rho_{n-1,k}$.  Also recall that for  $k \ge 1$, we  defined $\rho_{1,k}$ to be $S^{1}$.  For example, 
 \begin{align*} \rho_{1,2} &= S^1 \mbox{ and } \beta_{2,2} = S^{21} \\
 \rho_{2,3} &= S^{31} \oplus S^{21^2} \mbox{ and } \beta_{3,3} = S^{3^21} \oplus S^{321^2}.
\end{align*}

\begin{theorem}[Theorem~\ref{decomposition_thm}]  \label{decomposition_thm_2} Let $n,k \ge 1$.  Then as $\sg_{k(n-1)+1}$-modules,
\begin{equation}  \rho_{n,k} \cong  \beta_{n,k} \oplus \gamma_{n,k} , \end{equation}
for some  $\sg_{k(n-1)+1}$-module $\gamma_{n,k}$ all of whose irreducibles  have Young diagrams with at most $k-1$ columns.  
\end{theorem}

\begin{proof} Lemmas~\ref{hatlem1} and~\ref{hatlem2} together imply that as  $\sg_{k(n-2) +1}$-modules,
$$\rho_{n-1,k} \cong \bigoplus_{\substack{ \lambda \vdash k(n-1) +1\\   \lambda \, \mbox{\scriptsize has $k$ columns } } }c_\lambda S^{\lambda^-},$$  
where $c_\lambda$ is the multiplicity of $S^{\lambda}$ in $\rho_{n,k}$.
If we add a part of size $k$ to $\lambda^-$, we get $\lambda$.  Thus
$$ \beta_{n,k} \cong \bigoplus_{\substack{ \lambda \vdash k(n-1) +1\\  \lambda \, \mbox{\scriptsize has $k$ columns } }} c_\lambda S^{\lambda}.$$
The result now follows from
Proposition~\ref{lekprop}.
\end{proof}

\subsection{Another decomposition result}

We say that a bracketed permutation $[w_1,\dots,w_n]$ is a {\em non-comb} if either
\begin{itemize}
\item for some $i\ne j\in [n]$, the bracketed permutations $w_i$ and $w_j$ have more than one letter or
\item for some $i \in [n]$, the bracketed permutation $w_i$ is a non-comb.
\end{itemize}
These are the bracketed permutations that can't be obtained from an $n$-comb (\ref{combeq}) by just using the antisymmetry relations and not the generalized Jacobi relations.
For $ n,k \ge 1$, let $\tilde\rho_{n,k}$  be the submodule of $\rho_{n,k}$ spanned by the non-combs  in $\rho_{n,k}$.

The next result  shows that in order to obtain the decomposition of $\rho_{n,k}$ into irreducibles, we need only focus on the submodule spanned by non-combs.  This  will play a crucial role in the proof of Theorem~\ref{k=3} given in Section~\ref{ncsec}.  
\begin{theorem} \label{noncombth} For all $k,n \ge 1$, let $\tilde\rho_{n,k}$  be the submodule of $\rho_{n,k}$ spanned by  the non-combs  in $\rho_{n,k}$. Then
\begin{enumerate}
\item the multiplicity of $S^{k^{n-1}1}$ in $\rho_{n,k}$ is 1 
\item $\rho_{n,k} \cong S^{k^{n-1}1} \oplus \tilde\rho_{n,k}$
\item $\tilde\rho_{n,k} = 0$ if and only if $k=1,2$.
\end{enumerate}
\end{theorem}

\begin{proof} (1) follows immediately from Theorem~\ref{decomposition_thm_2} and  induction on $n$.  It is clearly true when $n=1$. 

We prove (2)
by constructing an $\sg_{k(n-1)+1} $-module isomorphism $$\psi:  S^{k^{n-1}1}  \to  \rho_{n,k} /\tilde \rho_{n,k}.$$
Let $t$ be the tableau of shape $k^{n-1}1$ whose first column from top down is 
$a_0,a_{1,1},\dots, a_{1,n-1}$ and whose  $i$th column, where $2 \le  i \le k$,   from top down is 
$a_{i,1}, a_{i,2}, \dots, a_{i,n-1}$.  We define $\psi$ on generators of $S^{k^{n-1}1} $ by letting
 $$\psi(\bar t)= [[\cdots [a_0,a_{1,1},\dots, a_{1,n-1}],a_{2,1},\dots, a_{2,n-1}] ,\dots ], a_{k,1},\dots, a_{k,n-1}].$$
We claim that this map  induces a well defined surjective $\sg_{k(n-1)+1} $-module homomorphism,
$\psi:  S^{k^{n-1}1}  \to  \rho_{n,k} /\tilde \rho_{n,k}$.  
To see this, first note that the column relations map to antisymmetry relations on brackets.  The Garnir relations $g_c^t$ map to  relations on combs that are the same as the relations  obtained by setting the non-combs equal to $0$ in the alternative Jacobi relation (\ref{filipeq}). Indeed, let $x_1$ in (\ref{filipeq}) be a comb and let $x_2,\dots x_n,y_1,\dots, y_{n-1}$  be single letters, making the left hand side of (\ref{filipeq}) a comb.  Then the summands of the right hand side of (\ref{filipeq}) are all combs except for the $i=1$ term, which is set equal to $0$.  
 Hence $\psi$ is well defined. Since, by Proposition~\ref{coneprop}, $\rho_{n,k} /\tilde \rho_{n,k}$ is spanned  by combs,  $\psi$ is surjective.

We claim that $\ker \psi=0$.  Suppose not. Then since $S^{k^{n-1}1}$ is irreducible, $\ker \psi=S^{k^{n-1}1}$.  It follows that $\rho_{n,k} = \tilde \rho_{n,k}$, which  means that the non-combs span $\rho_{n,k} $.    This implies that the non-combs of type $1^{k(n-2)+1}k$ 
span $\hat \rho_{n,k}$. 
  Since the   $b$-removing $\sg_{k(n-2)+1}$-module isomorphism $\varphi:\hat \rho_{n,k} \to \rho_{n-1,k}$ in the proof of (\ref{rho_eqn}) takes non-combs of  type $1^{k(n-2)+1}k$ to non-combs of type $1^{k(n-2)+1}$, the non-combs span $\rho_{n-1,k}$.  Continuing this argument leads to the conclusion that the non-combs span $\rho_{2,k} = \lie_{k+1}$, 
  which means that $\tilde \rho_{2,k} = \rho_{2,k}$.  
    
However, it follows from Pieri's rule that $\tilde \rho_{2,k}$ and $ \rho_{2,k}$ cannot be equal. Indeed,
   in any bracketed permutation $w$ of $\rho_{2,k}$, there are three kinds of $2$-brackets $[a,b]$, namely, those in which
  \begin{enumerate}
  \item only one of $a$ and $b$ are single letters in $[k+1]$
  \item both $a$ and $b$ are  single letters in $[k+1]$
  \item neither $a$ nor $b$ are single letters in $[k+1]$.
  \end{enumerate}
 For each $i = 1,2,3$, let $j_i(w)$ be the number of brackets in $w$ of the $i$th kind. Then $\sum_{i=1}^3 j_i(w) = k$.  The cyclic submodule $\rho_w$ of $\rho_{2,k}$ spanned by $\{\sigma w : \sigma \in \sg_{k+1}\}$  is isomorphic to a submodule of the induction product $\sgn_{1}^{\bullet j_1(w)} \bullet \sgn_2^{\bullet j_2(w)}$.  Hence by Pieri's rule, all irreducibles in $\rho_w$ have at most $j_1(w)+j_2(w)$ columns.  The bracketed permutation $w$ is a non-comb if and only if $j_3(w) \ge 1$. Thus if $w$ is a non-comb,  $j_1(w) + j_2(w) \le k-1$. Therefore, since $\tilde \rho_{n,k}$ is the sum of the cyclic submodules $\rho_w$, where $w$ is a non-comb,  all irreducibles in $\tilde \rho_{n,k}$ have at most  $k-1$ columns.  It follows that the multiplicity of $S^{k^11}$ in  $\tilde \rho_{2,k}$ is $0$, while by Part (1) of this theorem, the multiplicity of $S^{k^11}$  in $\rho_{2,k}$ is $1$, which means that  $\tilde \rho_{2,k}$ and $ \rho_{2,k}$ cannot be equal.  Thus the claim that $\ker \psi=0$ holds, from which we  conclude that $\psi$ is indeed an isomorphism.

To prove (3) first note that when $k=1$ or $k=2$, there are no non-combs; so $\tilde \rho_{n,k} = 0$.  For the converse, let  $\tilde \rho_{n,k} = 0$. Then by part (2), $\rho_{n,k} = S^{k^{n-1}1}$.  By  Theorem~\ref{decomposition_thm_2}, this implies $\rho_{2,k}=\lie_{k+1} \cong S^{k^11}$.  We use the Kraskiewicz-Weyman decomposition, Theorem~\ref{liek}, to show that $k=1$ or $k=2$.  Indeed,    $\rho_{2,3} = S^{31} \oplus S^{21^2}$. Thus $k\ne 3$. If $k \ge 4$, consider the standard Young tableau of shape $(k-1)2$ whose first row is $1,2,4,5,\dots, k$ and whose second row is $3,k+1$.  Since its major index is $2+k$, which is congruent to $1 \bmod k+1$, we have that $S^{(k-1)^12}$ is in $\rho_{2,k}$. Thus 
$k <4$.
\end{proof}

Note that Part (2) of the  theorem generalizes Theorem~\ref{k=2} since there are no non-combs when $k=2$.

\section{The non-comb submodule for $k=3$} \label{ncsec}

 In this section, we prove Theorem~\ref{k=3} by studying  the  submodule $\tilde\rho_{n,3}$ of $\rho_{n,3}$ spanned by the non-combs and then applying Theorem~\ref{noncombth}. This result can also be derived as an immediate consequence of   results for general $k$, namely,  the Kraskiewicz-Weyman decomposition (Theorem~\ref{liek}) and the Stabilization Theorem (Theorem~\ref{introstabth}), which is proved in the next two sections. We include the proof for $k=3$ here  because the proof of the Stabilization Theorem is substantially more difficult and because the $k=3$ case serves as a warmup for what is to follow.  Also this proof  is what motivated our proof of the Stabilization Theorem.

\begin{theorem}  \label{keq3th} For all $n \ge 2$,
$$\tilde\rho_{n,3} \cong_{\sg_{3n-2}} S^{3^{n-2}21^2}.$$
\end{theorem}

\begin{proof} Let $ V_{n,k}$ be the vector space of $n$-ary bracketed permutations on $[k(n-1)+1]$ subject only to the anticommuting relations, but not the generalized Jacobi relations.  Let $\tilde V_{n,k}$ be the subspace of  $ V_{n,k}$ generated by the bracketed permutations that are not combs.    Clearly $\tilde V_{n,k}$ is  invariant under the action of $\sg_{k(n-1)+1}$ on the multilinear component of the full vector space $V_{n,k}$.

Now let $k=3$. For each $\alpha \in \sg_{3n-2}$, let $v_{\alpha} $ be the bracketed permutation (subject only to the anticommuting relations),
$$ [[\alpha(1), \dots, \alpha(n)], [\alpha(n+1), \dots, \alpha(2n)], \alpha(2n+1),\dots,\alpha(3n-2)].$$
Clearly $\{v_\alpha: \alpha \in \sg_{3n-2}\}$ forms a of generating set  for $\tilde V_{n,3}$ and $\sigma \in \sg_{3n-2}$ acts on the generators by
$\sigma v_{\alpha} = v_{\sigma \alpha}$.  For an $\sg_i$-module $U$ and  an $\sg_j$-module $V$, let   
$U[V]$  denote the $\sg_{ij}$-module obtained by  inducing the $(\sg_j \wr \sg_i)$-module $U \otimes T^i(V)$, where $ \wr $ denotes  wreath product, and $T^i$ denotes the $i$th tensor power.  We will refer to $U[V]$ as the {\it composition product} of $U$ and $V$.
    Now note that as $\sg_{3n-2}$-modules,
\begin{equation}  \tilde V_{n,3} \cong  \sgn_{2}[\sgn_n] \bullet \sgn_{n-2}.\end{equation}

It is well known (see \cite[Section I8, Example 9]{Ma}) that the composition product 
$\sgn_{2}[\sgn_n]$ decomposes into irreducibles as follows
$$
\sgn_{2}[\sgn_n] \cong \bigoplus_{\begin{subarray}{c} i\in [n] \\ i \mbox{ \scriptsize{odd} } \end{subarray}} S^{2^{n-i} 1^{2i}}.
$$
Hence 
\begin{equation} \label{decompnoneq} \tilde V_{n,3} \cong \bigoplus_{\begin{subarray}{c} i\in [n] \\ i \mbox{ \scriptsize{odd} } \end{subarray}} S^{2^{n-i} 1^{2i}} \bullet \sgn_{n-2}.\end{equation}

For $\alpha \in \sg_{3n-2}$, we form the relation
$$R_\alpha: = 2(n-2) v_\alpha - \sum_{\begin{subarray}{c} i \in [2n] \\  j \in [3n-2]\setminus [2n]  \end{subarray}} v_{\alpha (i,j)}.$$
Let $R$ be the submodule of $\tilde V_{n,3}$ given  by $$R= \langle R_\alpha : \alpha \in \sg_{3n-2}\rangle.$$
Note that $R_\alpha$ is the sum of  $2(n-2)$ relations, namely,
$$v_\alpha - \sum_{i=1}^n  v_{\alpha (i,j)} \qquad \mbox{and} \qquad v_\alpha - \sum_{i =n+1}^{2n} v_{\alpha (i,j)}, \quad \mbox{ for each } j = 2n+1, \dots,  3n-2.$$
   By antisymmetry of the brackets these relations are alternative generalized Jacobi relations (\ref{filipeq}).  
Consequently,
\begin{equation} \label{modReq} \tilde \rho_{n,3} \subsetsim \tilde V_{n,3} / R, \end{equation} 
where $U \subsetsim V$ means $U$ is isomorphic to a submodule of $V$. 
We will show that the two modules are isomorphic to the irreducible
$S^{3^{n-2}21^2}$ (and are therefore  equal to each other).

Let $\tau$ be the linear operator on 
$\tilde V_{n,3} \cong \sgn_2[\sgn_n] \bullet \sgn_{n-2}$ that takes each generator $v_\alpha$ to $R_\alpha$. It is not difficult to see that this is a well defined
linear operator that restricts to a linear operator on   the direct summands of (\ref{decompnoneq}).   Let  $\tau_i$ be the restriction of $\tau$ to the direct summand $S^{2^{n-i} 1^{2i}} \bullet \sgn_{n-2}$.  Note that Lemma~\ref{indlem} can be applied since $\tau_i$ is equal to $\varphi_2^{\lambda_1,\lambda_2}$, where $\lambda_1 = 2^{n-i} 1^{2i} $ and $\lambda_2 = 1^{n-2}$.   For $\lambda_1$ and $\lambda_2$ to be compatible, $i$ must equal $0,1,2$, or $n$. Since $i$ is odd, we can eliminate $i=0,2$.  We can also eliminate $i=n$, since in this case, the number of columns of  $\lambda_1$ is $1$, which is less than $d=2$. This yields $$\ker \tau_1 =  S^{2^{n-1} 1^{2}\conc 1^{n-2}} = S^{3^{n-2}2 1^2}.$$
and
$\ker \tau_i =   (0)$ if $i \ne 1$.
Hence 
\begin{equation} \label{kereq}\ker \tau = \bigoplus_{\mbox{ \scriptsize odd } i \in [n]  } \ker \tau_i = \ker \tau_1 = S^{3^{n-2}2 1^2}.\end{equation}

 We claim that  
\begin{equation}  \im \tau= R.\end{equation}  Indeed, we have $\tau(v_\alpha) = R_\alpha$ for all $\alpha \in \sg_{3n-2}$.  Since the $ v_\alpha$ generate $\tilde V_{n,3}$ and the $R_\alpha$ generate $R$, the claim holds.

We  conclude from this and (\ref{kereq}) that
$$\tilde V_{n,3}/ R \cong S^{3^{n-2}21^2}.$$

Now by (\ref{modReq}) we  have that $\tilde \rho_{n,3} $ is either trivial or is isomorphic to  the irreducible $S^{3^{n-2}21^2}$.  The former case is impossible by Part~3 of Theorem~\ref{noncombth}.  Therefore $\tilde \rho_{n,3} \cong S^{3^{n-2}21^2}$, as desired.
\end{proof}

Now by Theorem~\ref{noncombth} we have,
\begin{corollary}[Theorem~\ref{k=3}] For all $n\ge 2$, as
 $\sg_{3n-2}$-modules,   $$\rho _{n,3} \cong S^{3^{n-1}1} \oplus S^{3^{n-2}21^2}.$$
 Consequently, by the hook length formula, $\dim \rho _{n,3} = \displaystyle \frac{4}{ \prod_{i=1}^3 (n+ i)} {3n \choose n,n,n}$.
\end{corollary}

\section{Proof of the Stabilization Theorem, Part 1: trees and quotients} \label{treequotsec}

In this section we present the first part of the proof of the Stabilization Theorem discussed in the introduction, which we restate here.
 For $k \ge 1$, set $\beta_{1,k} = \rho_{1,k} = S^1$.  For $n \ge 2$ and $k \ge 1$, let $\beta_{n,k}$ be the $\sg_{k(n-1)+1}$-module whose decomposition into irreducibles is obtained by adding a row of length $k$ to the top of each Young diagram in the decomposition of $\rho_{n-1,k}$.

\begin{theorem}[Theorem~\ref{introstabth} (Stabilization Theorem)] \label{stabth} Let $n \ge k \ge 1$.  Then   as $\sg_{k(n-1)+1}$-modules, 
\begin{equation}\label{stabeq} \rho_{n,k} \cong \beta_{n,k}.\end{equation} \end{theorem}

In light of Theorem~\ref{decomposition_thm_2},  equation (\ref{stabeq})  is equivalent to 
the assertion that  all the irreducibles contained in $\rho_{n,k}$ have exactly $k$ columns when $n \ge k$. In the previous section, we proved this for $3=k \le n$  by constructing the $\sg_{3(n-1)+1}$-module filtration 
 $$(0) \subset \tilde \rho_{n,3} \subset \rho_{n,3},$$ where
 $\tilde \rho_{n,3} $ is the submodule of $\rho_{n,3}$ spanned by the non-combs, and then showing that 
$\tilde \rho_{n,3}$ and $\rho_{n,3} / \tilde \rho_{n,3}$  are both irreducibles with  exactly $3$ columns; see Theorems~\ref{keq3th} and~\ref{noncombth}.  In this section,  with considerably more effort, we again employ the idea of working with quotients of submodules generated by bracketed permutations of  fixed ``shapes".

It is natural to view $n$-ary bracketed permutations as leaf-labeled plane rooted  $n$-ary trees.  Let $T$ be an (unlabeled) plane rooted $n$-ary tree with $k$ internal nodes and let $w$ be a labeling of the leaves of $T$ with distinct labels from set $A$, where $|A| = k(n-1) +1$.  Then the leaf-labeled tree $(T,w)$ encodes the bracketed permutation $[(T,w)]$ defined recursively as follows. 
If $k=0$ then $[(T,w)] = a$, where $A=\{a\}$, and if $k > 0$ then $[(T,w)] = [ [(T_1,w_1)],[(T_2,w_2)],\dots, [(T_n,w_n)]]$ where $(T_1,w_1),(T_2,w_2),\dots, (T_n,w_n)$ are the subtrees of the root of $(T,w) $ ordered from left to right.  We say that $ T$ is the {\em shape} of the bracketed permutation $[(T,w)]$. 

For example, in Figure~\ref{brtrfig} two leaf labeled plane rooted ternary trees $(T,w)$ are given with their associated bracketed permutation $[(T,w)]$ beneath them.  

\begin{figure}[h!]  \begin{center} \includegraphics[height=1.7in]{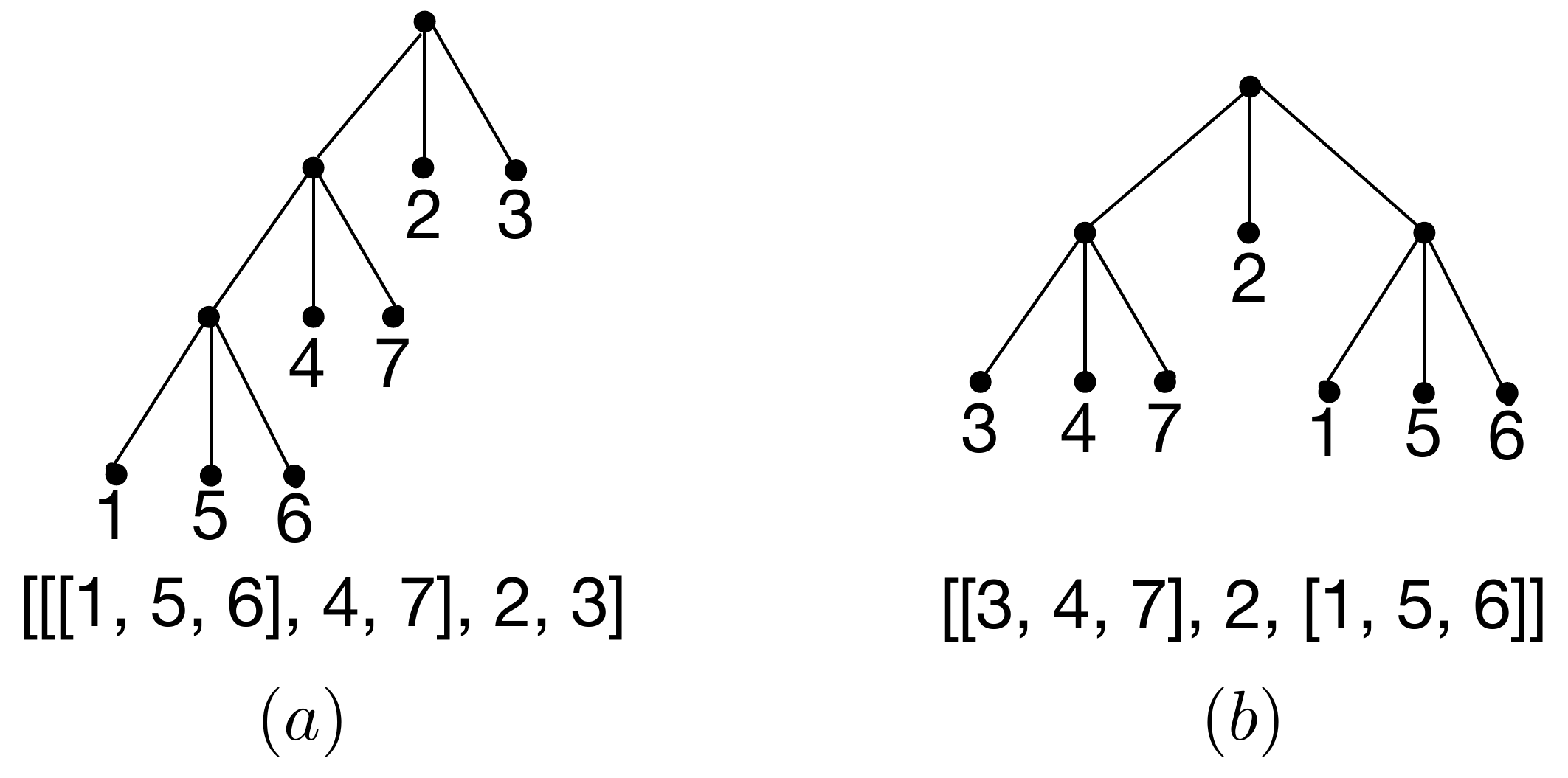} \end{center}
\caption{Leaf labeled ternary trees $(T,w)$ and corresponding bracketed permutations [(T,w)]}
\label{brtrfig}
\end{figure}

For $n,k \ge 1$, let $\mathbb T_{n,k}$ be the set of  plane  rooted $n$-ary trees  with $k$ internal nodes.  Note that for the leaf labeled trees $(T,w)$ in Figure~\ref{brtrfig}, the unlabeled tree $T$ belongs to $\mathbb T_{3,3}$.  
We start with a simple observation. 
 \begin{lemma}\label{mindeglem} Let $n \ge k \ge 1$ and $T \in \mathbb T_{n,k}$.  Then every internal node of $T$ has a  child that's a leaf. 
\end{lemma}
\begin{proof}Suppose $T$ has an internal node $a$ with no leaf-children.  Then the subtree
rooted at $a$ has at least $n+1$ internal nodes.  But since $k$ is the total number of internal nodes of $T$, we have $k \ge n+1 $ which contradicts the hypothesis.  
\end{proof}

Let $n,k \ge 1$. For each $T \in \mathbb T_{n,k}$, define the {\em depth vector} of $T$ by $$\delta(T) =(\delta_1(T),\delta_2(T),\dots, \delta_{d_T}(T)),$$ where $\delta_i(T)$ is the number of leaves of $T$ at depth $i$ and $d_T$ is the depth of $T$. (The depth of a node is the length of the path from the root to the node, and the depth of $T$ is the maximum depth over all nodes of $T$).   For the leaf labeled tree $(T,w)$ in Figure~\ref{brtrfig} (a), the  depth vector of the unlabeled tree $T$  is $\delta(T) = (2,2,3)$,  and for the tree $(T,w)$ in Figure~\ref{brtrfig} (b), the depth vector of the unlabeled tree  $T$ is $\delta(T)= (1,6)$.

We say that $T\in \mathbb T_{n,k}$  is {\em increasing} if for all  internal nodes $a, b$ of $T$, if $a$ is the parent of $b$ then  the number $\mu_T(a)$ of leaf-children of $a$  is less than or equal to the number $\mu_T(b)$ of leaf-children  of $b$.  In other words, $T$ is increasing if when the internal nodes $a$ of $T$ are labeled with $\mu_T(a)$ then the labels increase as we read down the tree. Examples  of  increasing and non-increasing trees   are given  in Figure~\ref{mutreefig}.

\begin{figure}[h!]  \begin{center} \includegraphics[height=1.7in]{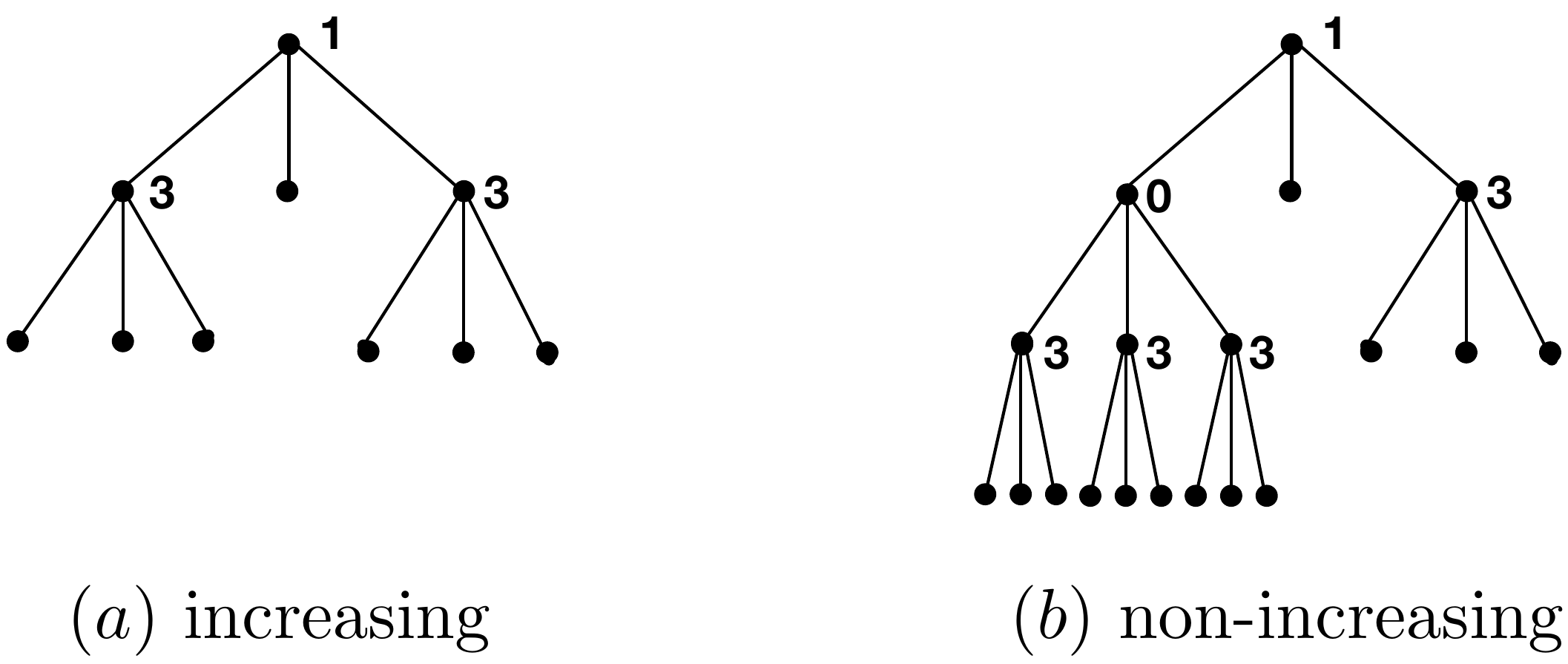} \end{center}
\caption{Internal nodes $a$ of $T$ labeled with $\mu_T(a)$}
\label{mutreefig}
\end{figure}

 \begin{ex}\label{combex} {\rm The {\em comb tree} $C_{n,k} \in \mathbb T_{n,k}$ is defined recursively as follows: if $k=1$ then $C_{n,k}$ is a root with $n$ children that are leaves and if $k>1$ then the subtree rooted at the left most child of the root of $C_{n,k}$ is $C_{n,k-1}$ and the other children of the root are leaves.   For example, the unlabeled tree $T$  in Figure~\ref{brtrfig} (a) is a comb, while the one in   Figure~\ref{brtrfig} (b) is not.  Note that $C_{n,k}$ is the shape of the $n$-comb defined in Section~\ref{Filsubsec}.  Also note that $C_{n,k}$ is increasing,
its depth vector $\delta(C_{n,k}) $ is $(n-1,n-1,\dots,n-1,n)$,
and $C_{n,k}$ has the lexicographically largest depth vector in $\mathbb T_{n,k}$.   }
\end{ex}

\begin{lemma} \label{depveclem}  
Any bracketed permutation of non-increasing shape $T \in \mathbb T_{n,k}$ can be expressed as a  sum (in $\rho_{n,k}$) of bracketed permutations whose shape $S$  satisfies
  $\delta(S) <_{lex} \delta(T)$, where $<_{lex}$ denotes lexicographical order.
 
 \end{lemma}

\begin{proof} 

Assume $[(T,w)]$ is a bracketed permutation such that $T \in \mathbb T_{n,k}$ is  non-increasing.  Then $T$  has internal nodes $a$ and $b$, where $a$ is the parent of $b$ and $\mu_T(a) > \mu_T(b)$.    When we apply the generalized Jacobi relation (\ref{type1}) to the leaf-labeled subtree of $(T,w)$ rooted at  $a$, we express $[(T,w)]$ as a  sum of bracketed permutations whose respective shapes $S$ are obtained from $T$ by exchanging  all the subtrees rooted at the children of $a$ other than $b$ with the subtrees rooted at all but one child  $x$ of $b$.  (See Figure~\ref{nondecfig} for an example.) This implies $\mu_S(a) = \mu_T(b) < \mu_T(a)$ if $x$ is an internal node, and $\mu_S(a) = \mu_T(b) -1< \mu_T(a)$ if $x$ is a leaf.  In either case, $\mu_S(a) < \mu_T(a)$, which implies that $\delta_j(S) < \delta_j(T)$, where $j-1$ is the depth of $a$.   Since $\delta_i(S) = \delta_i(T)$ for all  $i <j$, we have $\delta(S) <_{lex} \delta(T)$ as desired.
 \end{proof}
 
 \begin{figure}[h!]  \begin{center} \includegraphics[height=2.4in]{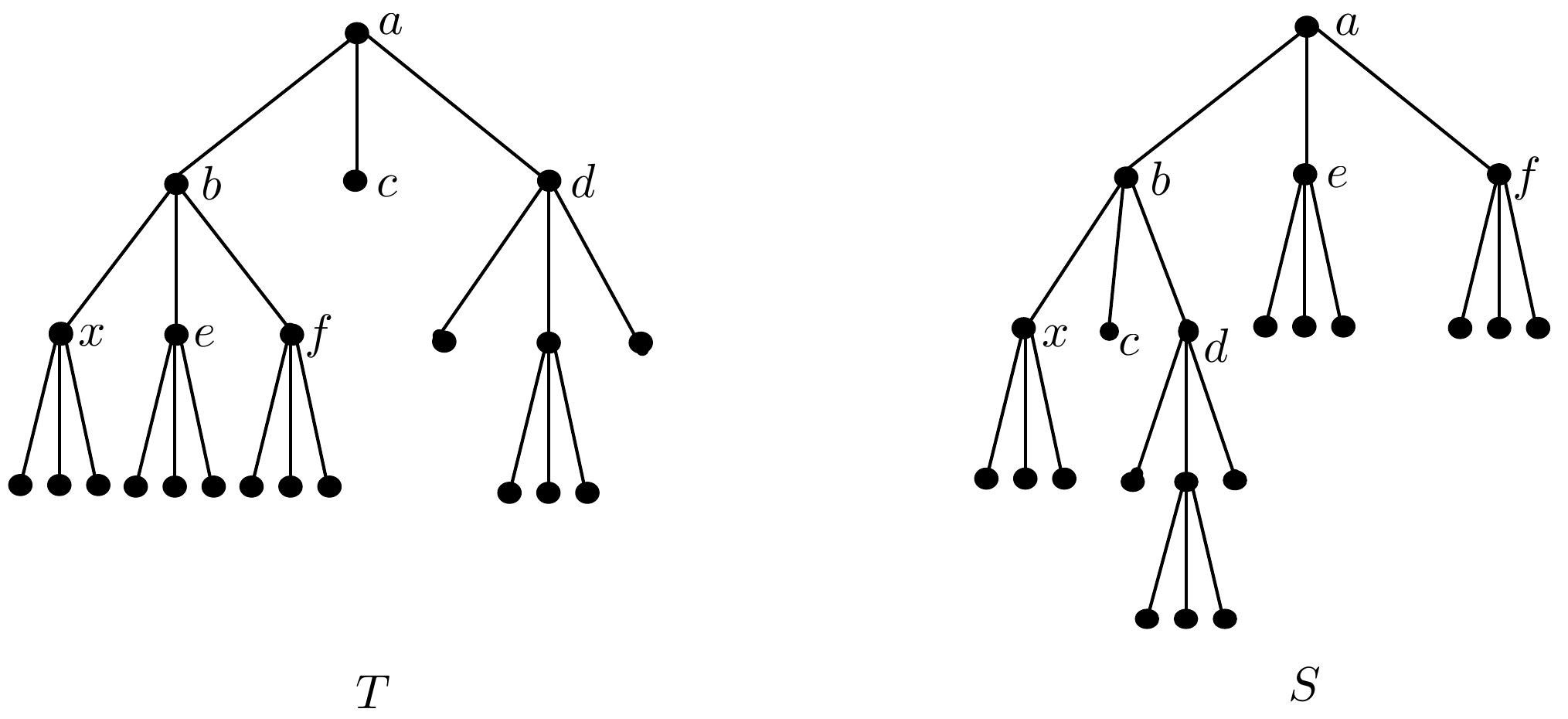} \end{center}
\caption{A tree  $S$ obtained from $T$ by applying generalized Jacobi relation (\ref{type1})  to $(T,w)$}
\label{nondecfig}
\end{figure}

The  lemmas above will enable us to reduce the 
 Stabilization Theorem (Theorem~\ref{stabth}) to the following key lemma.  For $T\in \mathbb T_{n,k}$, let $\rho_T$ denote the submodule of $\rho_{n,k}$ spanned by bracketed permutations of shape $T$. 
For any subset $\mathcal S \subseteq \mathbb T_{n,k}$, let $\rho_{\mathcal S}$ denote the submodule of $\rho_{n,k}$ spanned by bracketed permutations whose shapes are in $\mathcal S$.  For $T \in \mathbb T_{n,k}$, let $$D(T) =\{S \in \mathbb T_{n,k}: \delta(S) <_{lex}\delta(T)\}.$$

\begin{lemma}[Quotient Lemma] \label{quotlem}
Let  $T \in \mathbb T_{n,k}$ be an increasing tree for which each internal node has a child that's a leaf.  
   Then every irreducible contained in the  quotient module $\rho_T/(\rho_T \cap \rho_{D(T)})$ has exactly $k$ columns. 
\end{lemma} 

Section~\ref{trspecsec} is devoted to the  proof of the Quotient Lemma.  
We now use this lemma to prove the
following result, which by Theorem~\ref{decomposition_thm_2} is equivalent to the Stabilization Theorem.

\begin{theorem} \label{rhocolth} Let $n \ge k \ge 1$.   Each irreducible in $\rho_{n,k}$ has exactly $k$ columns.
\end{theorem}

\begin{proof}  
The result is equivalent to the following assertion:
For all $T \in  \mathbb T_{n,k}$, each irreducible  in $\rho_T$ has exactly $k$ columns.
We prove this assertion by induction on the depth vector $\delta(T)$, where the depth vectors are ordered lexicogaphically.  Since by Lemma~\ref{mindeglem}, every internal node of $T \in \mathbb T_{n,k}$ has a leaf child, we will be able to apply the Quotient Lemma (Lemma~\ref{quotlem}) to any increasing $T \in \mathbb T_{n,k}$.

 We start with the $\sg_{k(n-1)+1}$-module isomorphism 
\begin{equation} \label{isoeq} \rho_T \cong (\rho_T/(\rho_T \cap \rho_{D(T)}) )\oplus (\rho_T \cap \rho_{D(T)}).\end{equation}
The base case handles the trees with  minimum depth vector.  Let $T$ be such a tree.  Then $D(T)=\emptyset$ and by Lemma~\ref{depveclem},  $T$ is increasing.    The second summand of $(\ref{isoeq})$ vanishes and by applying the Quotient Lemma to the first summand, we get the desired result that each irreducible in $\rho_T$ has exactly $k$ columns. 

The induction step is handled with two cases.  Suppose $T$ does not have minimum depth vector.  Then $D(T) \ne \emptyset$.

 Case 1:  $T$ is increasing.  
By the Quotient Lemma (Lemma~\ref{quotlem}), all the irreducibles in the first  summand of (\ref{isoeq}) have exactly $k$ columns.  
Since $D(T) \ne \emptyset$, we can apply the induction hypothesis to the trees in $D(T)$, and conclude that all the irreducibles in  $\rho_{D(T)}$ have exactly $k$ columns, which implies that all the irreducibles in the second summand of (\ref{isoeq}) have exactly $k$ columns.

Case 2: $T$ is not increasing.  By Lemma~\ref{depveclem}, $\rho_T$ is a submodule of  a  sum $\sum_S \rho_S$, where all the trees $S$ have smaller depth vector.  By the induction hypothesis,  all the irreducibles of $\rho_S$ have exactly $k$ columns.  Hence the same is true for the sum and therefore for $\rho_T$. 
\end{proof}

\section{Proof of the Stabilization Theorem, Part 2: tree Specht modules} \label{trspecsec}
Recall that in the previous section we reduced the Stabilization Theorem to the Quotient Lemma.  In this section we prove the Quotient Lemma  by introducing two generalizations of the notion of Specht module. In addition to their application in the current work, we feel they might have some independent interest and be worthy of further study. 

\subsection{Tree Specht modules of the first kind} \label{trespecsubsec}

Let $T$ be an unlabeled  plane rooted tree and let $N$ be a positive integer.  (Note that we are not requiring $T$ to be an $n$-ary tree.)
A $T$-{\em partition} $\mu$ of  $N$ is a labeling of the nodes of $T$ with positive integers satisfying:
\begin{itemize}
\item $\sum_{a \in T} \mu(a) = N$
\item If $a$ is the parent of $b$ then $\mu(a) \le \mu(b)$.
\end{itemize} 
Note that if we view $T$ as the Hasse diagram of a poset $P$ in which the root  is the maximum element then the notion of $T$-partition of $N$ is the same as Stanley's $P$-partition of $N$; see \cite[Sec. 3.15.1]{St1}.  Note also that the $T$-partitions when $T$ is a path  (i.e., each internal node has exactly one child) are the same as ordinary partitions.

\begin{definition} \rm{Let $\mu$ be a $T$-partition of $N$.  A {\em $\mu$-tableau} $t$ is a filling of the nodes  of $T$ that satisfies
\begin{itemize}
\item each node $a \in T$ is filled with a column $Col(a)$ of distinct positive integers of length $\mu(a)$
\item the sets of entries of the columns are disjoint and their union is $[N]$
\end{itemize}
Let $\mathcal T_{T,\mu}$ denote the set of $\mu$-tableaux, where $\mu$ is a $T$-partition.}
\end{definition} 

 An example of a $T$-partition $\mu$  and a $\mu$-tableau $t \in \mathcal  T_{T,\mu}$ is given in Figure~\ref{treetabfig}.
In the example, we've oriented the tree from right to left with the root to the far right.  Note that top to bottom order matters in each column.  Note also that when $T$ is a path rooted at the right endpoint,   a $\mu$-tableau is an ordinary Young tableau of ordinary partition shape $\mu$.

\begin{figure}[h!]  \begin{center} \includegraphics[height=2.4in]{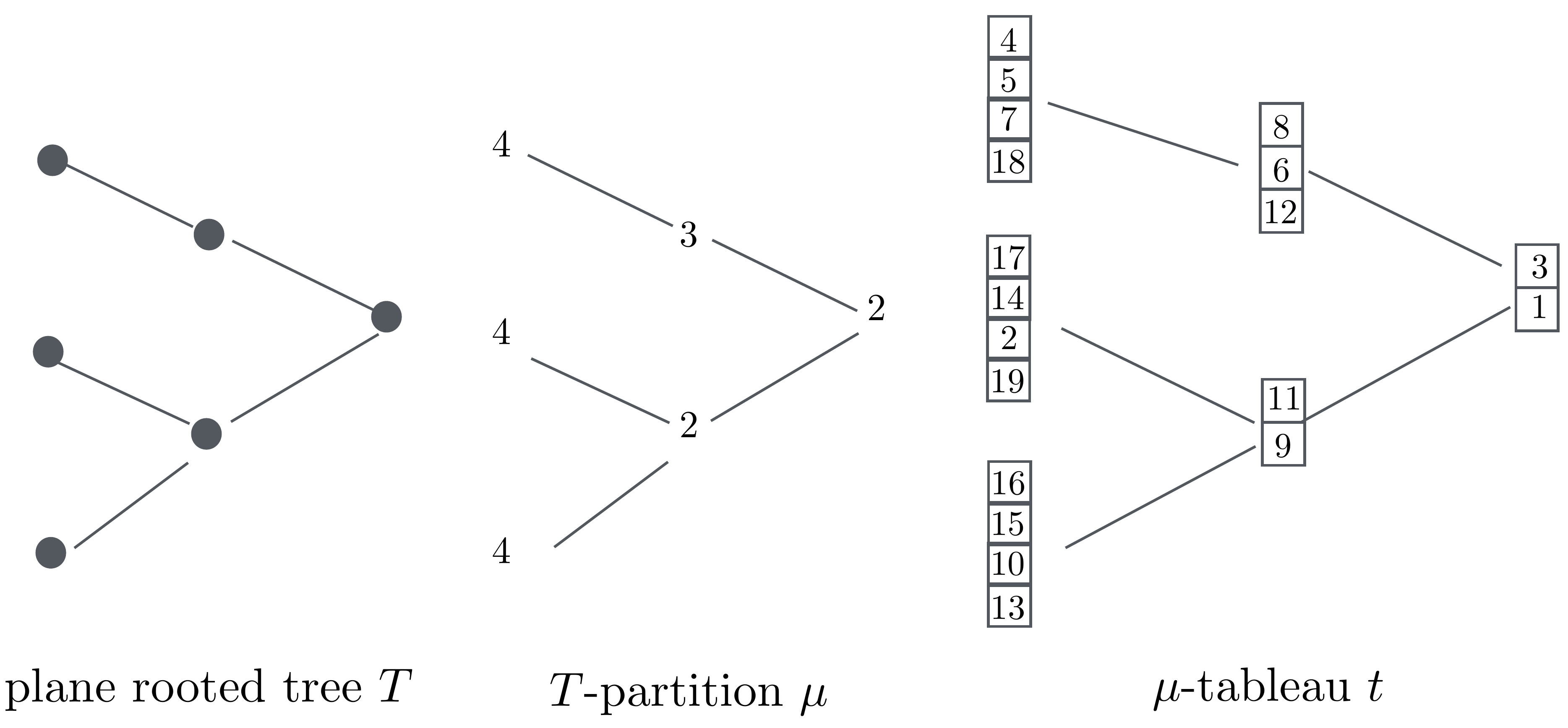} \end{center}
\caption{A tree partition and a tree  tableau.}
\label{treetabfig}
\end{figure}

 \begin{definition} \rm{Let $M^{T,\mu}$ be the vector space (over $\C$) generated by $\mathcal T_{T,\mu}$ subject to the {\em column relations}, which are of the form $t+s$, where $s \in \mathcal T_{T,\mu}$ is obtained from $t \in \mathcal T_{T,\mu}$ by switching two elements in the same column $Col(a)$ for any node $a \in T$.  Given $t \in \mathcal T_{T,\mu}$, let $\bar t$ denote the coset of $t$ in $M^{T,\mu}$.  We call these cosets  {\em  $\mu$-tabloids}.}
\end{definition}

A $\mu$-tableau is {\em column strict} if each column is increasing.  Clearly
$$\{\bar t : t \mbox{ is a column strict $\mu$-tableau}\}$$
is a basis for $M^{T,\mu}$.  
Note that $\sg_N$ permutes the entries  of the $\mu$-tableau $t \in \mathcal T_{T,\mu}$.  So $M^{T,\mu}$ is an $\sg_N$-module.  Note also that if we order the nodes $x_1,\dots,x_{|T|} $  of $T$ in some canonical way, say in preorder, we can view $\mu$ as a sequence $(\mu(x_1),\dots,\mu(x_{|T|})$ and $M^{T,\mu}$ as the induction product $$\sgn_{\mu(x_1)} \bullet \,\, \cdots \,\, \bullet \sgn_{\mu(x_{|T|})}.$$
This allows for a right action of $\sg_N$ as well as the left action described above.

\begin{definition}\label{gar1def} \rm{Let $\mu$ be a $T$-partition of $N$. The {\em tree Garnir relations of the first kind}  are defined by $$g^t_{a}:=
\bar t-\sum_s \bar s,$$ where  $a$ is a non-root node of $T$, $t \in  \mathcal T_{T,\mu}$, and the sum is over all $s\in \mathcal T_{T,\mu}$ obtained from $t\in \mathcal T_{T,\mu}$ by exchanging any
entry of  $Col(a)$  with the top entry of $Col(p(a))$, where $p(a)$ denotes the parent of $a$ in $T$. Let 
$G^{T,\mu}$ be the subspace of $M^{T,\mu}$ generated by the tree Garnir relations
$g^t_{a}$.  That is,
$$G^{T,\mu}:= \langle g_a^t : a \in T\setminus\{r\}, \, t\in \mathcal T_{T,\mu}\rangle, $$  where $r$ denotes the root of $T$.
The {\em tree Specht module of the first kind} $S^{T,\mu}$  is generated by the   $\mu$-tabloids  subject only to the tree Garnir relations of the first kind, i.e.,
$$S^{T,\mu} = M^{T,\mu}/ G^{T,\mu}.$$}
\end{definition}

Note that when $T$ is a path, the tree Garnir relations of the first kind are precisely the ordinary Garnir relations given in (\ref{gareq}) and the presentation defining the tree Specht module of the first kind is Fulton's  presentation for the ordinary Specht module given in Proposition~\ref{presentprop}.  Therefore if $T$ is a path then for all $T$-partitions~$\mu$,  
\begin{equation} \label{isoeq1} S^{T,\mu}\cong S^{\mu^*}, \end{equation} 
where  $\mu^*$ is the conjugate of $\mu$ viewed as an ordinary partition.  Note that the irreducible $S^{\mu^*}$ has exactly $|T|$ columns, where $|T|$ is the number of nodes of $T$.  We generalize this property  in the following result, whose proof is delayed to the next subsection.

\begin{theorem} \label{firstkindth} Let $\mu$ be a $T$-partition of $N$.  Then every irreducible in the $\sg_N$-module $S^{T,\mu}$ has exactly $|T|$ columns.
  \end{theorem}

We now describe the relationship between tree Specht modules and the main subject of this paper, namely $\rho_{n,k}$.  Let $n,k \ge 1$.   Given an increasing tree $T \in \mathbb T_{n,k}$  for which each internal node has a child that's a leaf, we prune $T$ by removing all its leaves to get the plane rooted tree $\hat T$.  The nodes  of $\hat T$ are the  internal nodes of $T$. For each such node $a$, recall that $\mu_T(a)$ denotes the number of leaf-children of $a$ in $T$.   Since each internal node of $T$ has a leaf-child, $\mu_T(a) \ge 1$ for all nodes $a$ of $\hat T$.
Since $T$ is increasing, $\mu_T$ is a $\hat T$-partition of $k(n-1)+1$.  
We now associate with each $\mu_T$-tableau $t \in \mathcal T_{\hat T,\mu_T}$, a  leaf labeling $w_t$ of $T$ defined by labeling the $i$th leaf-child  of $a$ (from left to right), where $a$ is an internal node of $T$ and $i \in [\mu_T(a)]$,  with the $i$th entry of $Col(a)$  of $t$ (from top down); see Figure~\ref{tabtolabfig}.  

\begin{figure}[h!]  \begin{center} \includegraphics[height=2.4in]{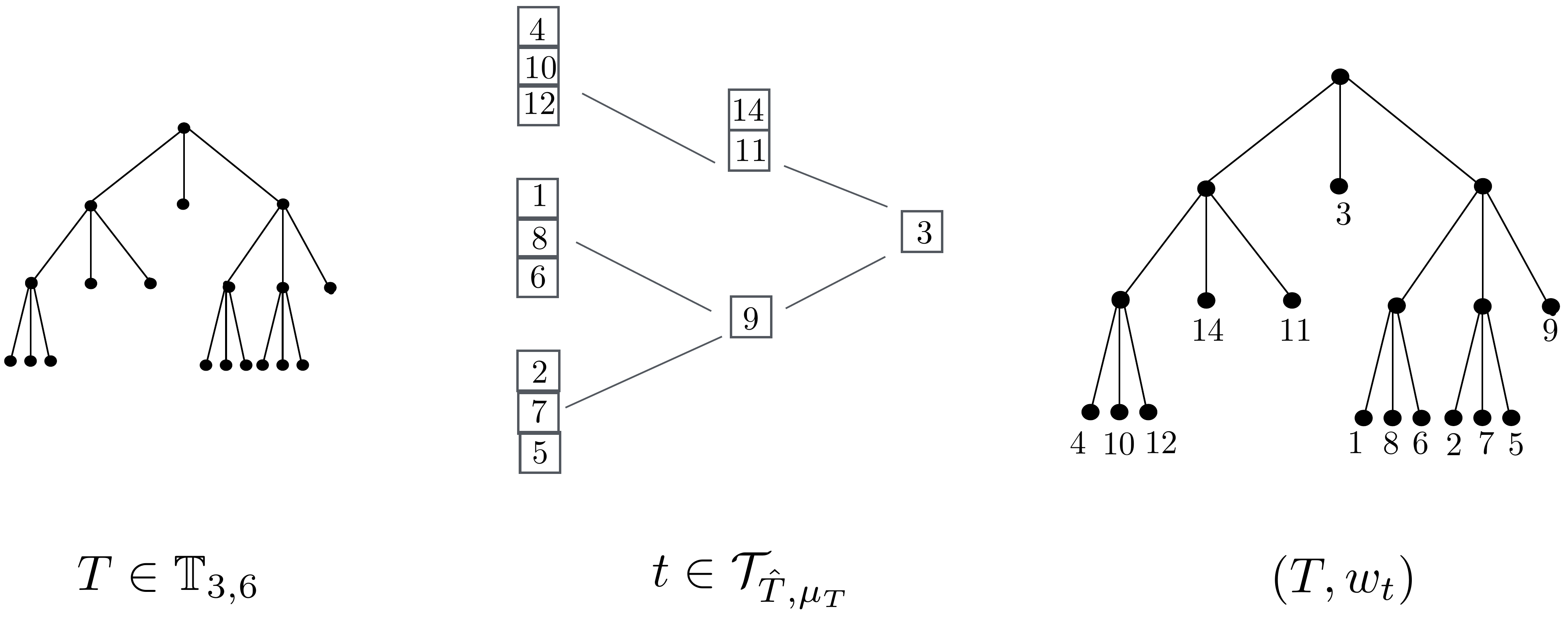} \end{center}
\caption{Leaf labeled tree $(T,w_t)$ associated with $\mu_T$-tableau $t$.}
\label{tabtolabfig}
\end{figure}

\begin{lemma} \label{treetablem} Let  $T \in \mathbb T_{n,k}$ be an increasing tree for which each internal node has a child that's a leaf.  Then the $\sg_{k(n-1)+1}$-module
$\rho_T/ (\rho_T \cap \rho_{D(T)})$ is isomorphic to a submodule of  $S^{\hat T, \mu_T}$.
\end{lemma}

\begin{proof}   Define the linear map $\varphi: M^{\hat T, \mu_T} \to \rho_T$, which takes the  $\mu_T$-tabloid $\bar t$ to the  bracketed permutation  $[(T,w_t)]$. 
 It follows from the antisymmetry of the bracket that this is a well defined $\sg_{k(n-1)+1}$-module  homomophism. Note also that the map is surjective.

Now consider the composition  $$\pi \circ \varphi: M^{\hat T, \mu_T} \to \rho_T/ (\rho_T \cap \rho_{D(T)}),$$ where   $\pi:\rho_T \to \rho_T/ (\rho_T \cap \rho_{D(T)})$ is the projection map.   We will show that the tree Garnir relations $g_a^t$, where $a$ is a nonroot node of $ \hat T$ and $t \in \mathcal T_{\hat T, \mu_T}$, satisfy
 \begin{equation} \label{treegareq}  g_a^t \in \ker (\pi \circ \varphi).\end{equation} 
 
  Indeed, first note that
$$\varphi(g_a^t) = \varphi(\bar t-\sum_s \bar s) = \varphi(\bar t)-\sum_s \varphi(\bar s) 
= [(T,w_t)] - \sum_u [(T,u)],$$ where  the summation index $s$ is as in Definition~\ref{gar1def} and the final sum is over all $u$ obtained from  $w_t$  by switching the label of a leaf child of $a$ with the label of the left most leaf    child $b$ of the  parent  $p(a)$ of $a$.   We claim that the alternative generalized Jacobi relation (\ref{filipeq}) applied to the subtree of $T$ rooted at  $p(a)$ yields  

\begin{equation} \label{inDeq}  [(T,w_t)]  -\sum_u [(T,u)] \in \rho_{D(T)} .\end{equation}
Indeed, the bracketed permutations on the right hand side of (\ref{filipeq}) that aren't one of the $[(T,u)] $'s have shapes $S$ that can be obtained from $T$ by switching the leaf $b$ with the subtree rooted at any child $c$ of $a$ that's an internal node; see Figure~\ref{derivfig}. Clearly, the number of leaves at  the depth of $b$ (or $a$) in $T$ is reduced by this switch, while nothing higher up in the tree changes.  Hence $\delta(S) <_{lex} \delta(T)$, which implies (\ref{inDeq}).  

Since $[(T,w_t)] - \sum_u [(T,u)]$ is clearly also in $\rho_T$, we have $\varphi(g_a^t) \in \rho_T \cap \rho_{D(T)}$, which means that  (\ref{treegareq}) holds.
Hence 
$$G^{\hat T,\mu_T} \subseteq \ker (\pi \circ \varphi).$$ Since $\rho_T/ (\rho_T \cap \rho_{D(T)})$ is the image of $\pi \circ \varphi$, it follows  that $\rho_T/ (\rho_T \cap \rho_{D(T)})$  is isomorphic to a submodule of $M^{\hat T,\mu_T}/G^{\hat T,\mu_T} = S^{\hat T,\mu_T}$.
\end{proof}

\begin{figure}[h!]  \begin{center} \includegraphics[height=1.8in]{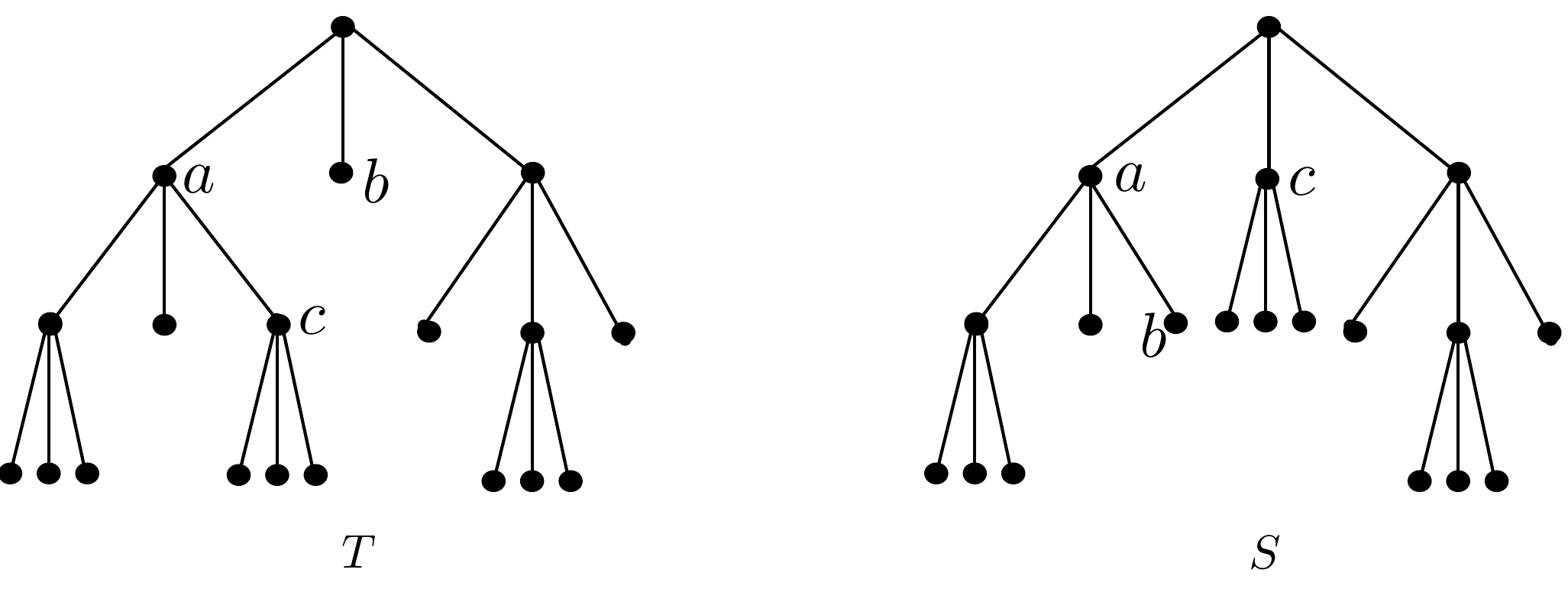} \end{center}
\caption{}
\label{derivfig}
\end{figure}

 \begin{remark} Lemma~\ref{treetablem} applied to the  comb tree $C_{n,k}$ given in Example~\ref{combex} asserts that $$\rho_{n,k}/\tilde \rho_{n,k} \subsetsim S^{k^{n-1}1}$$ since $\tilde \rho_{n,k} = \rho_{D(C_{n,k})}$ and by Proposition~\ref{coneprop},  $\rho_{n,k} = \rho_{C(n,k)} $.  By Theorem~\ref{noncombth}(2), this containment is in fact an isomorphism.
 \end{remark}
 
 We now use Theorem~\ref{firstkindth}, which will be proved in the next subsection, to prove the Quotient Lemma (Lemma~\ref{quotlem}).
 \begin{proof}[Proof of Lemma~\ref{quotlem}] It  follows from Lemma~\ref{treetablem}  that   each irreducible of $\rho_T/ (\rho_T \cap \rho_{D(T)})$ is an irreducible of $S^{\hat T, \mu_T}$.  By Theorem~\ref{firstkindth}, which we  have yet to prove, each irreducible of $S^{\hat T, \mu_T}$ has exactly $|\hat T|$ columns.  Since $|\hat T|$ is the number of internal nodes of $T$, which is $k$,  each irreducible of $\rho_T/ (\rho_T \cap \rho_{D(T)})$ has exactly $k$ columns.    
 \end{proof}

\subsection{Tree Specht modules of the second kind}  Our goal in this subsection is to prove Theorem~\ref{firstkindth} which states that all irreducibles in $S^{T,\mu}$ have exactly $|T|$ columns.  We carry this out by introducing a second generalization $\tilde S^{T,\mu}$ of Specht module, and showing that $S^{T,\mu}$ is isomorphic to a submodule of $\tilde S^{T,\mu}$ and that  all the irreducibles in $\tilde S^{T,\mu}$ have exactly $|T|$ columns.  

Our second generalization of Specht module differs from the first only in the Garnir relations.

\begin{definition} \rm{Let $\mu$ be a $T$-partition of $N$. The {\em tree Garnir relations of the second kind}  are defined by $$\tilde g^t_{a}:=
|D_a|  \mu(a) \bar t-\sum_s \bar s,$$ where    $t \in  \mathcal T_{T,\mu}$, $a$ is an internal node of $T$, the set $D_a$ consists of the  descendants of $a$, and the sum is over all $s\in \mathcal T_{T,\mu}$ obtained from $t\in \mathcal T_{T,\mu}$ by exchanging any
entry of  $Col(a)$  with any entry of $\bigcup_{b \in D_a} Col(b)$. Let $\tilde G^{T,\mu}$ be the subspace of $M^{T,\mu}$ generated by the tree Garnir relations
$\tilde g^t_{a}$. That is,
$$\tilde G^{T,\mu}:= \langle \tilde g_a^t : a \mbox{ an internal node of $T$, }  t\in \mathcal T_{T,\mu}\rangle. $$
The {\em tree Specht module of the second kind} $\tilde S^{T,\mu}$  is generated by the   $\mu$-tabloids  subject only to the tree Garnir relations of the second kind , i.e.,
$$\tilde S^{T,\mu} = M^{T,\mu}/ \tilde G^{T,\mu}.$$}
\end{definition}

Note that when $T$ is a path, the tree Garnir relations of the second kind are precisely the  Garnir relations given in (\ref{newgareq}) and the presentation defining the tree Specht module of the second kind is the presentation for the ordinary Specht module given in Theorem~\ref{newpresth}.  More precisely if $T$ is a path, then for all $T$-partitions $\mu$
\begin{equation} \label{iso2eq} \tilde S^{T,\mu}  \cong S^{\mu*},\end{equation}
where $\mu^*$ is the conjugate of $\mu$ viewed as an ordinary partition.

\begin{theorem} \label{col2th} Let $\mu$ be a $T$-partition of $N$.  Then every irreducible in the $\sg_N$-module $\tilde S^{T,\mu}$ has exactly $|T|$ columns and the last column has length $\mu(r)$ where $r$ is the root of $T$. 
\end{theorem}

\begin{proof} The result is obvious in  the case where $|T|=1$.  So assume $|T| > 1$.   Let $T_1,T_2, \dots,T_m$ be the subtrees rooted at the children of the root $r$ indexed from left to right.  The $T$-partition  $\mu$ restricted to $T_i$ defines a $T_i$-partition $\mu_i$. 
Let $N_i = \sum_{a \in T_i} \mu_i(a)$.  Clearly 
$ \sum_{i=1}^m N_i = N - \mu(r) $.

Let $Y$ be the  $\sg_{N-\mu(r)}$-module
$$Y= \tilde S^{T_1,\mu_1} \bullet \, \cdots \, \bullet \tilde S^{T_m,\mu_m}.$$  
Then $$\tilde S^{T,\mu} \cong (Y \bullet \sgn_{\mu(r)}) / \langle \tilde g_r^t : t \in \mathcal T_{T,\mu}\rangle.$$

Since  $\tilde S^{T_i,\mu_i} $ is isomorphic to a submodule of $M^{T_i,\mu_i} $ for each $i$,   by  Pieri's rule, each irreducible in $Y$ has at most $\sum_{i=1}^m |T_i| =|T|-1$ columns.   
By the definition of $ \varphi_d^{Y_1,Y_2}$ given in Lemma~\ref{indlem2}, $$\langle \tilde g_r^t : t \in \mathcal T_{T,\mu}\rangle = \im \varphi_{|T|-1}^{Y,\sgn_{\mu(r)}}.$$  It follows that $$\tilde S^{T,\mu} \cong \ker \varphi_{|T|-1}^{Y,\sgn_{\mu(r)}}.$$
Hence by Lemma~\ref{indlem2}, every irreducible in $\tilde S^{T,\mu} $ is of the form $S^{\lambda \conc 1^{\mu(r)}}$, where $S^\lambda$ is an irreducible of $Y$ with exactly $|T|-1$ columns. Therefore the irreducibles of $\tilde S^{T,\mu} $ have  exactly $|T|$ columns and the last column has length $\mu(r)$.
\end{proof}

Our final result  of this section relates the two types of tree Specht modules and thereby completes the proof of the Stabilization Theorem.  Recall that $U \subsetsim V$ means that $U$  that is isomorphic to a submodule of $V$.
\begin{theorem} \label{2kindth} Let $\mu$ be a $T$-partition of $N$.  Then $S^{T,\mu} \subsetsim \tilde S^{T,\mu}$ with isomorphism when $T$ is a path.
\end{theorem}

\begin{proof} The isomorphism follows from (\ref{isoeq1}) and (\ref{iso2eq}).  We prove the inclusion by induction on the number of nodes of $T$ and split the argument into two cases.

{\bf Case 1:} The root $r$ of $T$ has only one child $b$.  Let $T^\prime$ be the tree obtained from $T$ by removing $r$ and let $\mu^\prime$ be $\mu$ restricted to $T^\prime$.

We have 
 \begin{equation} \label{decSmueq} \tilde S^{T,\mu} = (\tilde S^{T^\prime, \mu^\prime} \bullet \sgn_{\mu(r)} )/ \langle \tilde g_r^t : t \in \mathcal T_{T, \mu}
\rangle \end{equation} 

By Theorem~\ref{col2th}, the decomposition of $\tilde S^{T^\prime, \mu^\prime}$  into irreducibles is of the form,
\begin{equation} \label{decompsec6eq} \tilde S^{T^\prime, \mu^\prime} \cong \sum_{\lambda \in D} \tilde d_\lambda S^\lambda,\end{equation} where $\tilde d_\lambda$ is a nonnegative integer and $D$ is the set  all Young diagrams $\lambda \vdash N- \mu(r)$ whose number of columns is $|T|-1$ and  last column length  is $\mu(b)$.  Therefore, since the virtual representations of the symmetric groups form a ring under  sum and induction product, by (\ref{decSmueq})
\begin{equation*} \tilde S^{T,\mu} \cong \sum_{\lambda \in D} \tilde d_\lambda (( S^\lambda \bullet \sgn_{\mu(r)}) /  \langle \tilde g_{|T|}^t : t \in \mathcal T_{\lambda\conc 1^{\mu(r)}}
\rangle),
\end{equation*} 
where $\tilde g_{|T|}^t$ is as  defined in  (\ref{newgareq}).
It now follows from  Theorem~\ref{newpresth} that 
\begin{align*} \tilde S^{T,\mu} &\cong \sum_{\lambda \in D} \tilde d_\lambda (( \tilde S^\lambda \bullet \sgn_{\mu(r)}) /  \langle \tilde g_{|T|}^t : t \in \mathcal T_{\lambda\conc 1^{\mu(r)}} \rangle).
\\ &\cong \sum_{\lambda \in D} \tilde d_\lambda  \tilde S^{\lambda\conc 1^{\mu(r)}}
\\ &\cong \sum_{\lambda \in D} \tilde d_\lambda   S^{\lambda\conc 1^{\mu(r)}}.
\end{align*} 

By  the induction hypothesis and (\ref{decompsec6eq}), we have 
that $$ S^{T^\prime, \mu^\prime} \cong \sum_{\lambda \in D} d_\lambda S^\lambda,$$ for some $d_\lambda$ satisfying $0 \le d_\lambda \le \tilde d_\lambda$. It follows that
\begin{align*} S^{T,\mu} &\cong S^{T^{\prime},\mu^\prime} \bullet \sgn_{\mu(r)}/\langle g_{b}^t : t \in \mathcal T_{T,\mu}
\rangle 
\\ &\cong \sum_{\lambda \in D} d_\lambda ((S^\lambda \bullet \sgn_{\mu(r)}) /  \langle g_{|T|-1}^t : t \in \mathcal T_{\lambda\conc 1^{\mu(r)}}
\rangle),\end{align*} 
 where $g_{|T|}^t$ is as  defined in (\ref{gareq}), since $\lambda \in D$ means $\lambda$ has $|T|-1$ columns and its last column has length $\mu(b) \ge \mu(r)$. 
By Fulton's presentation (Proposition~\ref{presentprop}) for Specht modules, we now have
\begin{equation*} \label{D1decompeq} S^{T,\mu} \cong  \sum_{\lambda\in D} d_\lambda S^{\lambda\conc1^{\mu(r)}}\subseteq \sum_{\lambda\in D} \tilde d_\lambda S^{\lambda\conc1^{\mu(r)}} \cong \tilde S^{T,\mu},\end{equation*}
which completes the proof for Case 1.

{\bf Case 2:} The root $r$ of $T$ has children $b_1,\dots,b_m$, where $m \ge 2$.  For each $i$, let $T_i$ be the subtree of $T$ rooted at $b_i$ and let $\mu_i$ be the restriction of $\mu$ to $T_i$.
Let $T_i^r$ be the subtree $T_i \cup \{r\}$ of $T$ and let $\mu_i^r$ be the restriction of $\mu$ to $T_i^r$. Let  $G^{T,\mu}_i$ be the submodule of $M^{T,\mu}$ generated by the Garnir relations of the first kind that involve $b_i$ or any of its descendants, i.e. $$G^{T,\mu}_i = \langle g^t_c: c \in T_i, t \in \mathcal T_{T,\mu}  \rangle .$$ 
Clearly $$\sum_{i=1}^m G^{T,\mu}_i = G^{T,\mu}$$   
It is not difficult to see that
\begin{equation} \label{subtreeeq1} M^{T,\mu}/G^{T,\mu}_i \cong M^{T_1,\mu_1} \bullet \,\, \cdots \,\, \bullet M^{T_{i-1},\mu_{i-1} }\bullet S^{T^r_i,\mu^r_i} \bullet  M^{T_{i+1},\mu_{i+1}} \bullet \,\, \cdots \,\, \bullet M^{T_{m},\mu_{m} }.\end{equation}

Now for $t \in \mathcal T_{T,\mu}$ and internal node $c \in T^r_i$, define
$$\tilde g_{c,i}^{t} := \begin{cases} \tilde g_c^t & \mbox{if } c \ne r \\
|T_i| \mu(r) \bar t - \sum_{s\in A_i} \bar s & \mbox{if } c = r,
\end{cases}$$
where $A_i$ is the set of  $\mu$-tableaux  that can be obtained from $t$ by switching any entry of  $Col(r)$ with any entry of $Col(a)$ for any node $a$ in $T_i$ other than the root $r$. Let $$\tilde G^{T,\mu}_i = \langle \tilde g^{t}_{c,i} : c \mbox{ is an internal node of }  T_i^r, t \in \mathcal T_{T,\mu} \rangle.$$  Note that
\begin{equation} \label{subtreeeq2} M^{T,\mu}/\tilde G^{T,\mu}_i  \cong   M^{T_1,\mu_1} \bullet \,\, \cdots \,\, \bullet M^{T_{i-1},\mu_{i-1} }\bullet \tilde S^{T^r_i,\mu^r_i} \bullet  M^{T_{i+1},\mu_{i+1}} \bullet \,\, \cdots \,\, \bullet M^{T_{m},\mu_{m} }.\end{equation} 
It follows from Case 1 or the induction hypothesis that $S^{T^r_i,\mu^r_i} \subsetsim \tilde S^{T^r_i,\mu^r_i}$.  Thus by (\ref{subtreeeq1}) and (\ref{subtreeeq2})   $$M^{T,\mu}/G^{T,\mu}_i   \subsetsim M^{T,\mu}/\tilde G^{T,\mu}_i .$$   This means that $\tilde G^{T,\mu}_i  \subsetsim G^{T,\mu}_i $ for all $i$.  We therefore have
\begin{equation} \label{Gcontaineq} \sum_{i=1}^m \tilde G^{T,\mu}_i  \subsetsim \sum_{i=1}^m  G^{T,\mu}_i  = G^{T,\mu} .\end{equation}
We claim  that for all $t \in \mathcal T_{T,\mu}$ and internal nodes $c \in T$,
$$\tilde g_c^t = \begin{cases} \tilde g^{t}_{c,i} & \mbox{if  }\,\, c\in T_i \\
\sum_{i=1}^m \tilde g^{t}_{r,i}  & \mbox{if } \,\,c= r.
\end{cases}
$$
Indeed, the first case follows immediately from the definitions.  For the second case,
we have 
\begin{align*} \sum_{i=1}^m \tilde g^{t}_{r,i} &= \sum_{i=1}^m \left( |T_i|  \mu(r) \bar t - \sum_{s \in A_i} \bar s \right) 
\\ &= \sum_{i=1}^m |T_i| \mu(r) \bar t - \sum_{i=1}^m \sum_{s \in A_i} \bar s
\\ &= (|T| - 1) \mu(r) \bar t - \sum_{s \in \bigcup A_i} \bar s
\\ &= \tilde g^t_r 
\end{align*} 

From this and (\ref{Gcontaineq}) we can conclude that $\tilde G^{T,\mu}  \subsetsim \sum_{i=1}^m \tilde G^{T,\mu}_i \subsetsim G^{T,\mu}$, which implies that
$$S^{T,\mu} = M^{T,\mu}/G^{T,\mu} \subsetsim M^{T,\mu}/\tilde G^{T,\mu}  = \tilde S^{T,\mu}.$$
Therefore the induction step is complete. 
\end{proof}

\begin{proof}[Proof of Theorem~\ref{firstkindth}] This is an immediate consequence of Theorems~\ref{2kindth} and~\ref{col2th}.
\end{proof}

\subsection{Recap of proof of Stabilization Theorem}

Let $n \ge k$ and let $T \in \mathbb T_{n,k}$ be an increasing tree.   By Lemma~\ref{mindeglem} every internal node of $T$ has a leaf child.  From Lemma~\ref{treetablem} and Theorem~\ref{2kindth}, we have the inclusions
$$\rho_T/ (\rho_T \cap \rho_{D(T)}) \subsetsim S^{\hat T, \mu_T} \subsetsim \tilde S^{\hat T, \mu_T},$$
where $\hat T$ is obtained from $T $ by removing its leaves.  
Since by Theorem~\ref{col2th}, the irreducibles in $\tilde S^{\hat T, \mu_T}$ have exactly $|\hat T|$ columns, the same is true  for $S^{\hat T, \mu_T}$ and for the quotient $\rho_T/ (\rho_T \cap \rho_{D(T)}) $.  Since $|\hat T|$ is equal to the number of internal nodes of $T$, which is $k$, the Quotient Lemma (Lemma~\ref{quotlem}) holds.  Recall that the only thing left to do in the   induction  proof of Theorem~\ref{rhocolth}, 
 was to prove the Quotient Lemma. Therefore Theorem~\ref{rhocolth} is now proved.  This implies that 
  $\gamma_{n,k} $ in Theorem~\ref{decomposition_thm_2} is $(0)$, which   means that $\rho_{n,k} = \beta_{n,k}$, as asserted by the Stabilization Theorem.

\section{The $k=4$ case and further considerations} \label{fursec}
\subsection{Decomposition for k=4} With the Stabilization Theorem now proved and the decomposition for $k=3$ established, we are also now able to fill in the  $k=4$ column of Table~\ref{decomptable}.  Using a computer program 
written in C++, we found that $\dim \rho_{3,4}= 1077$.  This is used in our proof of the following result.

\begin{theorem} \label{decomp4} For all $n \ge 3$, the $\sg_{4n-3}$-module $\rho_{n,4}$ decomposes as 
$$ \rho_{n,4} \cong    S^{4^{n-1}1} \oplus  S^{4^{n-2}32} \oplus   S^{4^{n-2} 31^2} \oplus    S^{4^{n-2} 2^21} \oplus S^{4^{n-2} 21^3}    \oplus S^{4^{n-3} 3^21^3} \oplus S^{4^{n-3}32^3} .$$ 
\end{theorem}

\begin{proof} We begin by proving the result for $n=3$, which is,
\begin{equation} \label{rho34eq}  \rho_{3,4} \cong    S^{4^21} \oplus  S^{432} \oplus   S^{431^2} \oplus    S^{42^21} \oplus S^{421^3}    \oplus S^{3^21^3} \oplus S^{32^3} .\end{equation}
 From the  Kraskiewicz-Weyman decomposition (Theorem~\ref{liek}) we obtain
$$\rho_{2,4}  \cong    S^{4^11} \oplus  S^{32} \oplus   S^{31^2} \oplus    S^{2^21} \oplus S^{21^3}.$$ It follows that
$$ \beta_{3,4} \cong S^{4^21}\oplus  S^{432} \oplus   S^{431^2} \oplus    S^{42^21} \oplus S^{421^3}.$$
By  the well-known hook length formula (c.f. \cite{Fu,St2}) we have $\dim \beta_{3,4} = 873$.  As was mentioned above, our computer calculation gave $\dim \rho_{3,4}= 1077$.  Hence    by Theorem~\ref{decomposition_thm_2}, 
$$\rho_{3,4} \cong S^{4^21}\oplus  S^{432} \oplus   S^{431^2} \oplus    S^{42^21} \oplus S^{421^3} \oplus  \gamma_{3,4},$$ where $\gamma_{3,4}$ is an $\sg_{9}$-module of dimension $204$ that decomposes into  irreducibles whose Young diagrams have fewer than $4$ columns.  We need only show that $\gamma_{3,4} \cong S^{3^21^3} \oplus S^{32^3}$, which by the hook length formula has dimension $204$.

By Proposition~\ref{coneprop}, we have that $\rho_{3,4}$ is isomorphic to a submodule of  the induction product $\rho_{3,3} \bullet \sgn_2$.  Therefore by Theorem~\ref{k=3},  $\rho_{3,4}$ is isomorphic to a submodule of $$(S^{3^21} \oplus S^{321^2}) \bullet \sgn_2 \cong S^{3^21} \bullet \sgn_2 \, \oplus \, S^{321^2} \bullet \sgn_2.$$  By Pieri's rule,   this module decomposes into $$\nu \oplus (S^{3^221} \oplus S^{3^21^3}) \oplus (S^{3^221} \oplus S^{3^21^3} \oplus S^{32^3} \oplus S^{32^21^2} \oplus S^{321^4}),$$
where $\nu$ is an $\sg_9$ module whose irreducibles have $4$ columns.  It follows that $\gamma_{3,4}$ is a submodule of 
$$2S^{3^221} \oplus 2S^{3^21^3} \oplus S^{32^3} \oplus S^{32^21^2} \oplus S^{321^4}.$$  By the hook length formula,
the dimensions of these irreducibles are $168, 120, 84, 108, 105$, respectively.  The only way to get a submodule of dimension $204$ from these irreducibles  is by taking $S^{3^21^3} \oplus S^{32^3}$.  Therefore 
$$\gamma_{3,4} = S^{3^21^3} \oplus S^{32^3},$$
which means that (\ref{rho34eq}) holds.

The result for general $n$ now follows from the Stabilization Theorem.
\end{proof}

\subsection{Above the diagonal} The Stabilization Theorem addresses what happens at or below the $n=k$ diagonal.  We conclude this section with an observation on what happens above the diagonal and a  conjectural converse, which generalizes  Conjecture~4.2 of \cite{FHSW1}.
 
 \begin{theorem} For all $n \le k$, if  $\rho_{n,k}$ has an irreducible with  exactly $j$ columns then  $n \le j \le k$.
 \end{theorem}
 
 \begin{proof}  We prove this by induction on $k$.    Suppose $\rho_{n,k}$ has an irreducible $S^\lambda$ with exactly $j$ columns.  Then by Proposition~\ref{lekprop}, $j \le k$.  
 
 We must now show that $n \le j$.  If $k=n$ then by Theorem~\ref{rhocolth}, we have $j =k =n$ as desired.  So assume $k >n$.  Since the combs span $\rho_{n,k}$ (Proposition~\ref{coneprop}), we have that $\rho_{n,k}$ is isomorphic to a submodule of the  induction product $\rho_{n,k-1} \bullet \sgn_{n-1}$.  It follows that  $S^{\lambda}$ is in the induction product.  Therefore by Pieri's rule, $\rho_{n,k-1}$ has an irreducible with $h$ columns where $h \in \{ j-1,j\}$.  Since $k >n$, we have $k-1 \ge n$, which means we can apply the induction hypothesis to  $\rho_{n,k-1}$ and conclude that $n\le h \le j $, as desired.
  \end{proof}

\begin{conjecture} If $n \le j \le k$ then there is an irreducible in $\rho_{n,k}$ that has exactly $j$ columns.  \end{conjecture}
  
  The conjecture is true when $n =2$.  This can be proved by using the result of Kraskiewicz and Weyman given in Theorem~\ref{liek}.  We leave this as an exercise for the reader.  The conjecture is also true when
  $1 \le k \le 4$ as one can see from our results displayed Table~\ref{knowtable}.  

\vspace{.1in} \no \large {\bf Acknowledgements}
\normalsize
\\
We thank Richard Stanley and Vic Reiner for helpful discussions.

\end{document}